\documentclass[11pt]{article}
\usepackage{verbatim}
\usepackage{amsthm}
\usepackage{mathtools}
\usepackage{caption}
\usepackage{multirow}
\usepackage{array}
\usepackage{empheq}
\usepackage{tabularx}
\usepackage{xcolor}
\usepackage[english]{babel}
\usepackage{blindtext}
\usepackage{longtable}
\usepackage[numbers]{natbib}
\usepackage{titlesec}
\usepackage[section]{placeins}
\usepackage[linesnumbered,ruled,vlined]{algorithm2e}
\usepackage{amsfonts,amssymb,amsmath,bm,a4,color,amsthm}

\renewcommand{\Re}{{\textup{Re}}}
\newtheorem{theorem}{Theorem}

\newtheorem{corollary}{Corollary}
\newtheorem{lemma}{Lemma}
\theoremstyle{remark}

\newtheorem{innercustomgeneric}{\customgenericname}
\providecommand{\customgenericname}{}
\newcommand{\newcustomtheorem}[2]{%
  \newenvironment{#1}[1]
  {%
   \renewcommand\customgenericname{\textup{\textbf{#2}}}%
   \renewcommand\theinnercustomgeneric{\textbf{##1}}%
   \innercustomgeneric
   \itshape
  }
  {\endinnercustomgeneric}
}

\newcustomtheorem{customtheorem}{Theorem}
\newcustomtheorem{customlemma}{Lemma}

\renewcommand{\arraystretch}{1.5}
\begin{document}

\title{Explicit estimates for the logarithmic derivative and the reciprocal of the Riemann zeta function}
\author{Nicol Leong}
\date{}

\maketitle

\begin{abstract}
In this article, we give explicit bounds of order $\log t$ for $\sigma$ close to $1$, for two quantities: $|\zeta'(\sigma +it)/\zeta(\sigma +it)|$ and $|1/\zeta(\sigma +it)|$. We correct an error in the literature, and especially in the case of $|1/\zeta(\sigma +it)|$, also provide improvements in the constants. Using an argument involving the trigonometric polynomial, we additionally provide a slight asymptotic improvement within the classical zero-free region: $1/\zeta(\sigma +it) \ll (\log t)^{11/12}$. The same method applied to the Korobov--Vinogradov zero-free region gives a new record: the unconditional bound $1/\zeta(\sigma +it) \ll (\log t)^{2/3}(\log\log t)^{1/4}$.
\end{abstract}

\section{Introduction}\label{intro}
In the study of the Riemann zeta function $\zeta(s)$ and related functions, one often requires upper bounds for the logarithmic derivative of $\zeta(s)$ and the reciprocal of $\zeta(s)$. For example, these come up in the Perron formula in the process of estimating the Chebyshev function $\psi(x)=\sum_{n\le x}\Lambda(n)$, where $\Lambda(n)$ is the von Mangoldt function, or $M(x)=\sum_{n\le x} \mu (n)$, where $\mu(n)$ is the M\"{o}bius function. The reader is referred to \cite{trudgian2015explicit}, \cite{chalker}, \cite{Dudek_16}, \cite[(3.14)]{titchmarsh1986theory} for more information on this.

The chief purpose of this article is to prove explicit upper bounds for $\zeta'(s)/\zeta(s)$ and $1/\zeta(s)$. Along the way, we also provide in \S\ref{preliminaries} a collection of other improved bounds on quantities involving $\zeta(s)$, which might be useful for anyone seeking to do explicit computations in the area.

Our main results will be presented in detail in \S\ref{main results section}, but let us first give a brief history of results on the two quantities $\zeta'(s)/\zeta(s)$ and $1/\zeta(s)$. Trudgian in \cite{trudgian2015explicit} gives an explicit version of all the lemmas in \cite[\S3.9]{titchmarsh1986theory}, which were originally due to Landau. This then allowed Trudgian to obtain explicit bounds on $\zeta'(s)/\zeta(s)$ and $1/\zeta(s)$ of order $O(\log t)$. This was then extended in \cite{hiaryleongyangArxiv} and \cite{cully2024explicit} to explicit bounds of order $O(\log t/\log \log t)$. However, estimates for the two quantities mentioned above tend to have large constants, often rendering them unsuitable for practical use. It is a non-trivial task to reduce these constants. This is due to the fact that the factor of $\zeta(s)$ in the denominator means that such estimates can only be valid in a zero-free region, and the constants rapidly increase in size as we approach the boundary of the region, i.e., $\zeta(s)\to 0$.

While the bounds on $|\zeta'(s)/\zeta(s)|$ given in this paper might seem worse compared to \cite{trudgian2015explicit}, we take the opportunity to point out an error in that paper. In the quantity $A$ of \cite[Lemma~4]{trudgian2015explicit}, there should be a factor of $1/c$ instead of $1/2c$. This leads to the overall constants being better than they should be. For instance, in Table $2$ of said paper, which shows bounds of the form $|\zeta'(s)/\zeta(s)| \le R_1 \log t$ for $\sigma \ge 1-1/W\log t$, one should have $(W,R_1) =(6,538.16)$ and $(12,59.49)$, instead of $(6,382.58)$ and $(12,40.14)$, respectively. This also impacts the subsequent bounds on $|1/\zeta(s)|$ ($R_2$ in \cite{trudgian2015explicit}), which are derived from $R_1$. In this article, we have corrected the erroneous bounds and even slightly improved upon the (corrected) bounds.

Furthermore, by adapting a method of \cite[\S 5]{cully2024explicit} and being more careful in our derivation on bounds of $|1/\zeta(s)|$, we obtain improvements in some (erroneously better than they should be) bounds in \cite{trudgian2015explicit}. In addition, when $\sigma=1$, the final assertion of Corollary \ref{cor main 1/zeta} is significantly better than \cite[Proposition A.2]{carneiro2022optimality}. 

Finally, a novel argument utilising a combination of the classical non-negative trigonometric polynomial and a uniform bound of the form $\zeta(s) \ll (\log t)^{2/3}$ for $\sigma \ge 1$, allows us to get an asymptotic improvement in the case of $1/\zeta(s)$. In particular, working within zero-free regions of the forms
\begin{equation*}
    \sigma \ge 1-\frac{1}{Z_1 \log t},\qquad \sigma \ge 1-\frac{\log\log t}{Z_2 \log t}, \qquad  \sigma \ge 1-\frac{1}{Z_3 (\log t)^{2/3}(\log\log t)^{1/3}},
\end{equation*}
for some positive constants $Z_1,\,Z_2,\,Z_3$, one usually proves upper bounds for $1/\zeta(s)$ of the order
\begin{equation}\label{usual orders}
    \log t, \qquad \frac{\log t}{\log\log t}, \qquad (\log t)^{2/3}(\log\log t)^{1/3},
\end{equation}
respectively. Without making any changes to the zero-free region used, our new method would allow one to refine the results in \eqref{usual orders}, to
\begin{equation}\label{usual orders 1}
     (\log t)^{11/12}, \qquad  \frac{(\log t)^{11/12}}{(\log\log t)^{3/4}}, \qquad  (\log t)^{2/3}(\log\log t)^{1/4},
\end{equation}
respectively. As far as we are aware, results of these types were not previously known. In this paper, we prove an explicit version of the first bound in \eqref{usual orders 1}. The third bound in \eqref{usual orders 1} is the best unconditional bound currently known, and an explicit version of it is given in \cite{leeleongmobius2024}.

\section{Main results}\label{main results section}

Here we address matters of notation before listing some of our main results. We note that the arguments used in this article are largely those of \cite{cully2024explicit}.

\subsection{Properties of the Riemann zeta function}\label{properties}

In what follows, let $s=\sigma+it$, with $\sigma,\,t \in \mathbb{R}$, and define $H:= H_0 - 1/2$, where $H_0 := 3.\,000\,175\,332\,800 \cdot 10^{12}$ is the \textit{Riemann height}, i.e., the height to which the Riemann hypothesis has been verified (see \cite{platt2021riemann}). Our arguments will consider $t\le H$ and $t \ge H$ as separate cases. 

The classical zero-free region for $\zeta(s)$ states that for $t\ge 2$,
\begin{equation*}
\zeta(s)\neq 0, \qquad \text{ for }\qquad \sigma\ge 1-\frac{1}{W_0\log t}.
\end{equation*}
Currently, the best known result is due to \cite{mossinghoff2024explicit}, where the value $W_0=5.558691$ is permissible. We take $W_0$ to be this value.

We denote the \textit{Riemann hypothesis up to height} $T$ by $RH(T)$, by which we mean: for all non-trivial zeros $\rho = \beta +i\gamma$ such that $\zeta(\rho)=0$ and $\gamma \le T$, we have $\beta = 1/2$. Similarly we denote the \textit{Riemann hypothesis} by $RH$, by which we mean all non-trivial zeros $\rho$ have $\beta =1/2$.

\subsection{Bounds on $|\zeta'(s)/\zeta(s)|$}

Here we present our main results, which are corollaries of the theorems in \S\ref{section zeta'/zeta} and \S\ref{section 1/zeta}.

\begin{corollary}\label{cor main zeta'/zeta}
Let $W> W_0$. In a region
\begin{equation*}
    \sigma \ge 1-\frac{1}{W\log t}
\end{equation*}
where $\zeta(\sigma+it)\neq 0$, we have for $t\ge 13$, 
\begin{equation*}
\left|\frac{\zeta'(\sigma+it)}{\zeta(\sigma+it)}\right|\le Q \log t,
\end{equation*}
where $(W,Q)=(10,71.220)$ is admissible, and other values of $Q$ with corresponding $W$ are presented in Table \ref{table:Q}.
\end{corollary}
\begin{corollary}\label{main zeta'/zeta for 1}
For $\sigma \ge 1$ and $t\ge 13$, we have
\begin{equation*}
\left|\frac{\zeta'(\sigma+it)}{\zeta(\sigma+it)}\right|\le 24.303 \log t.
\end{equation*}
\end{corollary}
\begin{corollary}\label{cor zeta'/zeta RH}
In the region
\begin{equation*}
\sigma \ge \sigma_0 >1/2 \qquad \text{ and } \qquad t_0\le t \le H,
\end{equation*}
we have
\begin{equation*}
\left| \frac{\zeta'(\sigma+it)}{\zeta(\sigma+it)}\right| \le Q_{RH(H_0)}\log t,
\end{equation*}
where $(\sigma_0,\,t_0,\,Q_{RH(H_0)})=(0.8,14,23.759)$ is admissible, and other values for $Q_{RH(H_0)}$ with corresponding values of $\sigma_0$ and $t_0$ are given in Table \ref{table:QRH}. Furthermore, by assuming $RH$, the upper bound restriction on $t$ can be removed.  
\end{corollary}

\subsection{Bounds on $|1/\zeta(s)|$}

\begin{corollary}\label{cor main 1/zeta}
Let $W> W_0$. In a region
\begin{equation*}
    \sigma \ge 1-\frac{1}{W\log t}
\end{equation*}
where $\zeta(\sigma+it)\neq 0$, we have for $t\ge 13$, 
\begin{equation*}
\left|\frac{1}{\zeta(\sigma+it)}\right|\le Y \log t,
\end{equation*}
where $(W,Y) =(13,1719)$ is admissible, and other values of $Y$ with corresponding $W$ are presented in Table \ref{table:Y}.

In addition, we have for $\sigma\ge 1$ and $t\ge 13$,
\begin{equation*}
\left| \frac{1}{\zeta(\sigma+it)}\right| \le 30.812\log t,
\end{equation*}
while for $\sigma=1$ and $t\ge 2$,
\begin{equation*}
\left| \frac{1}{\zeta(1+it)}\right| \le 29.388\log t.
\end{equation*}
\end{corollary}
\begin{corollary}\label{cor main 1/zeta asymp}
Let $W> W_0$. In a region
\begin{equation*}
    \sigma \ge 1-\frac{1}{W\log t}
\end{equation*}
where $\zeta(\sigma+it)\neq 0$, we have for $t\ge 13$, 
\begin{equation*}
\left|\frac{1}{\zeta(\sigma+it)}\right|\le Y'(\log t)^{11/12},
\end{equation*}
where $(W,Y)=(13,4904)$ is admissible, and other values of $Y'$ with corresponding $W$ are presented in Table \ref{table:Y'}.

In addition, we have for $\sigma\ge 1$ and $t\ge 13$,
\begin{equation*}
\left| \frac{1}{\zeta(\sigma+it)}\right| \le 87.725(\log t)^{11/12},
\end{equation*}
while for $\sigma=1$ and $t\ge 2$,
\begin{equation*}
\left| \frac{1}{\zeta(1+it)}\right| \le 81.900(\log t)^{11/12}.
\end{equation*}
\end{corollary}

\section{Some preliminaries}\label{preliminaries}
In this section, we present the results used in obtaining our estimates. The argument is largely that of \cite{hiaryleongyangArxiv} and \cite{cully2024explicit}, in which detailed proofs of the core lemmas can be found. At the heart of the method lies the Borel--Carath\'{e}odory Theorem (see \cite[Theorem~4]{hiaryleongyangArxiv}), by which the following two lemmas are proven. We remark that Lemma \ref{log_deriv_zeta_lem1} is used in the process of proving Lemma \ref{lem_3.2_HB}.

\begin{theorem}[Borel--Carath\'{e}odory]\cite[Theorem~14]{hiaryleongyangArxiv}\label{borelcaratheodory}
Let $s_0$ be a complex number. 
Let $R$ be a positive number, possibly depending on $s_0$. 
Suppose that the function $f(s)$ is analytic in a region containing
the disc $|s-s_0|\le R$. Let $M$ denote 
the maximum of $\textup{Re}\, f(s)$ on the boundary $|s-s_0|=R$. 
Then, for any $r\in (0,R)$
and any $s$ such that $|s-s_0|\le r$,
\begin{equation*}
|f(s)| \le \frac{2r}{R-r}M + \frac{R+r}{R-r}|f(s_0)|.
\end{equation*}
If in addition $f(s_0)=0$, then 
for any $r\in (0,R)$ and any $s$ such that $|s-s_0|\le r$,
\begin{equation*}
|f'(s)| \le \frac{2R}{(R-r)^2}M.
\end{equation*}
\end{theorem}

\begin{lemma}\cite[Lemma~15]{hiaryleongyangArxiv}\label{log_deriv_zeta_lem1}
Let $s_0$ be a complex number.
Let $r$ and $\alpha$ be positive numbers (possibly depending on $s_0$) 
such that $\alpha <1/2$.
Suppose that the function $f(s)$ is analytic in a region containing the disc $|s-s_0|\le r$. Suppose further there is a number $A_1$ independent of $s$ such that,
$$\left|\frac{f(s)}{f(s_0)}\right| \le A_1,\qquad |s-s_0|\le r.$$
Then, for any $s$ in the disc $|s-s_0|\le \alpha r$
we have
\begin{equation*}
\left| \frac{f'(s)}{f(s)} - \sum_\rho \frac{1}{s-\rho}\right| \le \frac{4\log A_1}{r(1 -2\alpha)^2},
\end{equation*}
where $\rho$ runs through the zeros of $f(s)$ in the disc
$|s -s_0|\le \tfrac{1}{2} r$, counted with multiplicity.
\end{lemma}

\begin{lemma}\cite[Lemma~16]{hiaryleongyangArxiv}\label{lem_3.2_HB}
Let $s$ and $s_0$ be complex numbers with real parts $\sigma$ and $\sigma_0$, respectively.
Let $r$, $\alpha$, $\beta$, $A_1$ and $A_2$ be positive numbers, 
possibly depending on $s_0$, such that $\alpha<1/2$ and $\beta<1$.
Suppose that the function $f(s)$ satisfies 
the conditions of Lemma \ref{log_deriv_zeta_lem1} with $r$, $\alpha$ and $A_1$, 
and that
$$\left| \frac{f'(s_0)}{f(s_0)}\right| \le A_2.$$ 
Suppose, in addition, that $f(s)
\neq 0$ for any $s$ in the intersection of the disc $|s-s_0|\le r$ and the right half-plane $\sigma \ge \sigma_0 - \alpha r$. 
Then, for any $s$ in the disc $|s-s_0|\le \alpha\beta r$,
\begin{equation*}
\left| \frac{f'(s)}{f(s)}\right| \le \frac{8\beta\log A_1}{r(1-\beta)(1-2\alpha)^2} + \frac{1+\beta}{1-\beta}A_2.
\end{equation*}
\end{lemma}

The following corollary is useful for proving results in a right-half plane, under $RH(T)$ or $RH$.

\begin{corollary}\label{cor_3.1_HB_F-M_RH}
If in the disc $|s-s_0|\le r$, we have that $f(s)$ is analytic, non-zero, and satisfies $\left|\tfrac{f(s)}{f(s_0)}\right| \le A_1$, then for $0< \alpha_0 <1$,
\begin{equation*}
\left| \frac{f'(s)}{f(s)} \right| \le \frac{2 \log A_1}{r(1-\alpha_0)^2}, \qquad |s-s_0|\le \alpha_0 r.
\end{equation*}
\end{corollary}
\begin{proof}
We follow a similar proof to that of Lemma \ref{log_deriv_zeta_lem1} (also \cite[Lemma~15]{hiaryleongyangArxiv}). Define $g(s)=\log (f(s)/f(s_0))$, where the logarithm branch is determined by $g(s_0)=0$. Furthermore, since $f(s)\neq 0$ and is analytic in the disc $|s-s_0|\le r$, hence $g(s)$ is analytic there.

By hypothesis, we have $|f(s)/f(s_0)|\le A_1$ in $|s-s_0|\le r$, so in the same region,
\begin{equation*}
\textup{Re}(g(s))=\log\left|\frac{f(s)}{f(s_0)} \right| \le \log A_1.
\end{equation*}
To complete the proof, notice that $g'(s)= f'(s)/f(s)$, and apply Theorem \ref{borelcaratheodory} in the two discs $|s-s_0|\le \alpha_0 r < r$ to obtain
\begin{equation*}
|g'(s)|= \left| \frac{f'(s)}{f(s)}\right| \le \frac{2r\log A_1}{(r-\alpha_0 r)^2} = \frac{2 \log A_1}{r(1-\alpha_0)^2},
\end{equation*}
for $|s-s_0|\le \alpha_0 r$.
\end{proof}

Next we require some uniform bounds for $\zeta(s)$. Lemma \ref{plpbound} slightly improves \cite[Corollary~1]{trudgian2016improvements2}, which is proven using a generalisation of the Phragm\'{e}n--Lindel\"{o}f Principle, due to Trudgian \cite[Lemma~3]{Trudgian_14}. 

Unfortunately, Trudgian's proof contains an error, the consequence of which is that his result is true only if one includes the restriction $\alpha_2 =\beta_2$ (according to the notation of \cite[Lemma~3]{Trudgian_14}). In Lemma \ref{plp} below, we present a corrected version which includes this restriction. We refer the reader to \cite{fiori2025notephragmenlindeloftheorem} for an in-depth discussion on this error and a strategy to recovering the original result\footnote{The author thanks Andrew Fiori for bringing this to our attention.}.

\begin{lemma}\label{plp}
Let $a,b,Q$ be real numbers such that $b>a$ and $Q+a>1$. Let $f(s)$ be a holomorphic function on the strip $a\le \textup{Re}(s) \le b$ such that
\begin{equation*}
|f(s)|< C\exp(e^{k|t|})
\end{equation*}
for some $C>0$ and $0<k<\tfrac{\pi}{b-a}$. Suppose further that there are $A,B,\alpha_1,\alpha_2, \beta_1, \beta_2 \ge 0$ such that $\alpha_1 \ge \beta_1$, $\alpha_2 =\beta_2$, and
\begin{equation*}
|f(s)| \le
\begin{cases}
A|Q+s|^{\alpha_1}(\log |Q+s|)^{\alpha_2} \quad \text{for } \textup{Re}(s)=a ;  \\ 
B|Q+s|^{\beta_1}(\log |Q+s|)^{\beta_2} \quad \text{for } \textup{Re}(s)=b . 
\end{cases}
\end{equation*}
Then for $a\le \textup{Re}(s) \le b$, one has
\begin{equation*}
|f(s)|\le \Big({A|Q+s|^{\alpha_1}|\log (Q+s)|^{\alpha_2}}\Big)^{\tfrac{b-\textup{Re}(s)}{b-a}}\Big({B|Q+s|^{\beta_1}|\log (Q+s)|^{\beta_2}}\Big)^{\tfrac{\textup{Re}(s)-a}{b-a}}.
\end{equation*}
\end{lemma}

\begin{lemma}\label{plpbound}
Let $h(s)=(s-1)\zeta(s)$, and let $\delta_r$ be a positive real number. Furthermore, let $Q_0 > 1/2$ be a number for which
\begin{align}\label{h cond}
|h(\tfrac{1}{2} +it)| &\le k_1|Q_0+\tfrac{1}{2} +it|^{k_2 +1}(\log|Q_0 +\tfrac{1}{2} +it|)^{k_3} \\
|h(1+\delta_r +it)| &\le k_4|Q_0+1+\delta_r +it|(\log|Q_0 +1+\delta_r +it|)^{k_3},  \nonumber
\end{align}
for all $t$. Then for $\sigma\in[\tfrac{1}{2},1+\delta_r]$ and $t\ge t_0 \ge e$, we have
\begin{align}\label{plp3}
|\zeta(s)| \le& \left(k_4\right)^{\frac{\sigma-0.5}{0.5+\delta_r}}(1+a_0(1+\delta_r,Q_0,t_0))^{k_3} \log^{k_3} t \\
&\boldsymbol{\times}(1+a_1(1+\delta_r,Q_0,t_0))\left( k_1 (1+a_1(1+\delta_r,Q_0,t_0))^{k_2}t^{k_2}\right)^{\frac{1+\delta_r -\sigma}{0.5+\delta_r}}  ,\nonumber 
\end{align}
where
\begin{equation*}
a_0(\sigma,Q_0,t)=\frac{\sigma+Q_0}{2t^2 \log t}+\frac{\pi}{2\log t}+\frac{\pi(\sigma + Q_0)^2}{4t \log^2 t}, \qquad a_1(\sigma,Q_0,t)=\frac{\sigma+Q_0}{t}.
\end{equation*}
In addition, the right-hand side of \eqref{plp3} is decreasing in $\sigma$, provided that $k_i \ge 0$, and the following hold:
\begin{align}\label{tcond0}
t \ge \left( \frac{k_4}{k_1 (1+a_1 (1+\delta_r,Q_0,t_0))^{k_2}}\right)^{1/k_2}. 
\end{align}
\end{lemma}
\begin{proof}
The proof is analogous to that of \cite[Lemma~2.1]{trudgian2016improvements2}, through a direct application of Lemma \ref{plp} with $a=1/2$ and $b=1+\delta_r$, except that we are more precise in estimating \eqref{plp3} by retaining the dependence on $\sigma$ in the exponents of $1+a_1(\sigma,Q_0,t)$ and $1+a_0(\sigma,Q_0,t)$. 

Note that different to the proof of \cite[Lemma~2.1]{trudgian2016improvements2}, we use only two lines, $\sigma =1/2$ and $\sigma = 1+\delta_r$, omitting the need for a bound at $\sigma =1$. This is because we intend to find the maximum bound over $\sigma\in [1/2,1+\delta_r]$, which we show occurs at $\sigma=1/2$, as our bound is decreasing in $\sigma$. The decreasing condition can be verified by checking when the derivative of the right-hand side of \eqref{plp3} is negative, and this is what gives rise to the condition in \eqref{tcond0}. Additionally, we found that further splitting the interval complicates \eqref{tcond0} and raises the threshold of $t$ for the lemma to be valid.
\end{proof}

\begin{corollary}\label{plp_cor}
For $\sigma\in[\tfrac{1}{2},1+\delta_r]$ and $t\ge t_0\ge e$, we have
\begin{align*}
|\zeta(s)| \le 0.618 (1+a_0(1+\delta_r,1.31,t_0))(1+a_1(1+\delta_r,1.31,t_0))^{7/6}t^{1/6}\log t,
\end{align*}
provided that
\begin{equation}\label{tcond}
t^{1/6}\log t \ge \frac{\zeta(1+\delta_r)}{0.618 (1+a_1 (1+\delta_r,1.31,t_0))^{1/6}},
\end{equation}
where $\delta_r$, $a_0 (\sigma,1.31,t)$, and $a_1 (\sigma,1.31,t)$ are defined as in Lemma \ref{plpbound}.
\end{corollary}
\begin{proof}
Recall that Hiary, Patel, and Yang \cite{hiarypatelyang2022} have $|\zeta(1/2 +it)|\le 0.618t^{1/6}\log t$, for $t\ge 3$, while a quick calculation on Mathematica shows that $|\zeta(1/2 +it)| \le 1.461$ for $|t|<3$ (see also \cite{hiary2016explicit}). Meanwhile, one has by $|\zeta(s)|\le \zeta(\sigma)$ that for all $t$,
$$|\delta_r +it||\zeta(1+\delta_r+it)|\le \frac{\zeta(1+\delta_r)}{\log t}|\delta_r +it|\log t. $$
Therefore, it can be verified that both bounds in \eqref{h cond} are met.

We now refer to Lemma \ref{plpbound}, taking the maximum of \eqref{plp3} at $\sigma=1/2$. It has been shown in \cite[p.~15]{yang2023explicit} that $Q_0 = 1.31$ is permissible in \eqref{h cond} together with
\begin{equation*}
(k_1,k_2,k_3,k_4,Q_0) = \left(0.618,\tfrac{1}{6},1,\tfrac{\zeta(1+\delta_r)}{\log t},1.31 \right),
\end{equation*}
hence we are done.
\end{proof}
\begin{lemma}\label{1 line 2/3 bd}
For every $|t|\ge 3$, we have
\begin{equation*}
    |\zeta(1+it)|\le  58.096(\log|t|)^{2/3}.
\end{equation*}
\end{lemma}
\begin{proof}
    We use same argument in the first paragraph of \cite[p. 260]{trudgian2014new}, which modifies Ford's proof in \cite[Lemma 7.3]{ford_2002_zeta-bounds}. 
    
    The difference is that we now apply the modified lemma to the proof of Theorem $1.1$ of \cite{bellotti2024explicit} (which itself uses the abovementioned lemma of Ford's). 
    
    The improvements then come from applying the modified \cite[Lemma 7.3]{ford_2002_zeta-bounds} with new parameters $C= 8.7979$ and $D= 132.94357$ (see \cite[\S 4]{bellotti2024explicit}), in the limiting case of $t\ge 10^{108}$ (see \cite[\S 5]{bellotti2024explicit}). This proves our lemma.
\end{proof}
\begin{lemma}\label{zeta_log_bound 2/3}
For $t \ge t_0 \ge 3$ and $\sigma\ge 1$, we have
\begin{equation*}
|\zeta(\sigma + it)| \le C_{3}(\log t)^{2/3} ,
\end{equation*}
with $a_0$ defined as in Lemma \ref{plpbound}, and
\begin{equation}\label{zeta_c3}
C_3 = 58.096\sqrt{1 + 9t_0^{-2}}\left(1+ a_0(2,1,t_0)\right)^{2/3}.
\end{equation}
\end{lemma}
\begin{proof}
The proof is similar to that of \cite[Lemma 7]{hiaryleongyangArxiv}. First, to show that the result holds for $1 \le \sigma \le 2$, we use the Phragm\'en--Lindel\"of Principle, i.e., Lemma \ref{plp}, on the holomorphic function 
$$f(s) = (s - 1)\zeta(s).$$ 
To this end, on the $1$-line we have
\[
|f(1 + it)| = |t|\,| \zeta(1 + it)| \le 58.096\, |2 + it|(\log |2 + it|)^{2/3}.
\]
This inequality is verified numerically via Mathematica for $|t| < 3$, and is a consequence of Lemma \ref{1 line 2/3 bd} for $|t| \ge 3$. 

Also, on the line $s=2+it$, the trivial bound $|\zeta(s)|\le \zeta(\sigma)$ gives
\[
|f(2 + it)| \le |1 + it|\, \zeta(2) < 58.096|3 + it|(\log |3 + it|)^{2/3},
\]
for all real $t$. 
Now apply Lemma \ref{plp} with the following parameter values: $a = 1$, $b = 2$, $Q = 1$, $\alpha_1 = \beta_1 = 1$, $\alpha_2=\beta_2 = 2/3$, and $A =B=58.096$. We obtain
\begin{equation}\label{pl bound}
    |f(s)| \le 58.096|1 + s||\log (1 + s)|^{2/3},
\end{equation}
for $1 \le \sigma \le 2$. Moreover, for the same range of $\sigma$, we have the following estimates:
\begin{equation*}
    \begin{split}
        \frac{|1 + s|}{|s - 1|} &\le \frac{\sqrt{9 + t^2}}{t} \le \sqrt{1 + 9t_0^{-2}},\\
        |\log (1 + s)| &\le \left(1+ a_0(2,1,t_0)\right) \log t,
    \end{split}
\end{equation*}
where $a_0$ is defined as in Lemma \ref{plpbound}. The second inequality above is derived using the same argument as that in \cite[\S2]{trudgian2016improvements2}.

Thus for $1 \le \sigma \le 2$ and $t \ge t_0$, 
 the inequality \eqref{pl bound} implies that
\begin{align*}
|\zeta(\sigma + it)| &\le 58.096\sqrt{1 + 9t_0^{-2}}\left(1+ a_0(2,1,t_0)\right)^{2/3}\log^{2/3}t
\le C_{3} \log^{2/3}t,
\end{align*}
as desired, where $C_3$ is defined in \eqref{zeta_c3}.

Finally, we address the case where $\sigma \ge 2$. Here since $\sigma >1$, we apply the trivial bound
\begin{equation}\label{zeta trivial bd}
   |\zeta(s)|\le \zeta(\sigma)\le \frac{\sigma}{\sigma -1}, 
\end{equation}
where $\sigma/(\sigma-1)$ is decreasing in $\sigma$. Then for $\sigma \ge 2$, it is easy to verify that
\begin{equation*}
    |\zeta(\sigma+it)|\le 2 \le C_3 (\log t)^{2/3}
\end{equation*}
for all $t\ge t_0$.
\end{proof}
\begin{lemma}\cite[Lemma 4]{Simonic22}\label{aleks lemma}
In the region $1/2 < \sigma \le 3/2$ and $0\le t \le 2\exp(e^2)$, we have
\begin{equation*}
\left| \frac{1}{\zeta(\sigma+it)}\right| \le \frac{4}{\sigma-(1/2)}.
\end{equation*}
\end{lemma}

In the remainder of this section, we present some estimates concerning the Riemann zeta function, including some for $\sigma >1$, which are better than the trivial ones in some finite range. 

We would like to point out that Lemmas \ref{zeta_real_lower_bound_0} and \ref{zeta'_zeta_real_upper_bound}, as well as Corollary \ref{zeta'_zeta_real_upper_bound cor}, do not contribute to the main results of this paper, but are included as they are potentially useful estimates for anyone seeking to do explicit computations involving $\zeta(s)$.

Similarly, Lemma \ref{zeta_bound} and Corollary \ref{cor:zeta-bound}, which can both be viewed as an extension of Corollary \ref{propA1} inside the critical strip, are not actually required for the main purposes of this paper. However, we include them here as the proofs are similar to that of Corollary \ref{propA1}, and at the same time we provide a brief discussion on their usefulness.
Lemma \ref{zeta_bound} is an improvement on \cite[Proposition 5.3.2]{jameson_2003}, which proves a simlar result.

\begin{lemma}\label{zeta_bound}
Let $W>0$ be a real constant. Then in the region
\begin{equation*}\label{sig reg 2}
1-\frac{1}{W \log t} \le \sigma \leq \sigma_1,
\end{equation*}
we have for $t\ge t_0\ge 3$ and any $\eta\in[2/t_0,1-1/t_0]$ that
\begin{equation*}
|\zeta(\sigma+it)| \le  e^{1/W}\log t + C_0(W,\sigma_1,t_0,\eta)
\end{equation*}
where
\begin{align*}
    C_0(W,\sigma_1,t_0,\eta) =& e^{1/W} \Bigg( \log\left(\eta+\frac{1}{t_0}\right) + \gamma + \frac{1}{12\eta^2} \\
&+ \frac{1}{t_0} \left(\frac{1}{6\eta^2} + \frac{1}{\eta} + 1 \right)+\frac{16(\sigma_1^2+2\sigma_1-1)+3}{192\eta^2t_0^2} \Bigg).
\end{align*}
\end{lemma}

\begin{proof}
We will use the following representation for $\zeta(s)$ for $N\ge 2$,
\begin{equation}\label{backlund}
\zeta(s) = \sum_{n=1}^{N-1} \frac{1}{n^s} + \frac{1}{2N^s}+\frac{N^{1-s}}{s-1}+\frac{s}{12N^{s+1}}-\frac{s(s+1)}{2}\int_N^\infty \frac{\varphi^*(u)}{u^{s+2}}\text{d}u,
\end{equation}
where $\varphi^*(u)$ is the periodic function obtained by extending the polynomial $\varphi(u)=u^2-u+1/6$ on the interval $[0,1]$ (for full details see e.g., \cite[Equation (8)]{backlund1916nullstellen}). Note that \eqref{backlund} holds for $\sigma>0$ by analytic continuation (see \cite[Theorem~3.5]{titchmarsh1986theory}).

In the region of $\sigma$ considered, we have 
\begin{equation}\label{N_bound}
|N^{-s}|=N^{-\sigma}\le \exp\left(-\Big( 1-\frac{1}{W\log t}\Big)\log N\right) = N^{-1}\exp\left( \frac{\log N}{W\log t}\right).
\end{equation}
This bound can be used throughout (\ref{backlund}). The first term in \eqref{backlund} can also be bounded using the estimate in \cite[Lemma~2.1]{francis2021investigation}, so for $N\ge 2$,
\begin{align*}
\left| \sum_{n=1}^{N-1} \frac{1}{n^s} + \frac{1}{2N^s}\right|  &=  \left| \sum_{n=1}^{N} \frac{1}{n^s} - \frac{1}{2N^s}\right| \\
&\le \exp\left( \frac{\log N}{W\log t}\right) \left( \sum_{n=1}^{N} \frac{1}{n} + \frac{1}{2N} \right) \\ 
&\le \exp\left( \frac{\log N}{W\log t}\right) \left(\log N + \gamma + \frac{1}{N} -\frac{1}{12N^2}+\frac{1}{64N^4} \right).
\end{align*}
The term $N^{1-s}/(s-1)$ will be bounded by $t^{-1} \exp\left( \frac{\log N}{W\log t}\right)$. The integral in (\ref{backlund}) can be evaluated using the bound $|\varphi^*(u)|\leq 1/6$, so
\begin{align*}
    \int_N^\infty \frac{|\varphi^*(u)|}{u^{\sigma+2}}\text{d}u \leq \frac{1}{6(\sigma+1)} \frac{1}{N^{\sigma+1}} \leq \frac{1}{6} \exp\left( \frac{\log N}{W\log t}\right) \frac{1}{N^2}.
\end{align*}
This implies
\begin{align*}
    \left|\frac{s}{12N^{s+1}}-\frac{s^2+s}{2}\int_N^\infty \frac{\varphi^*(u)}{u^{s+2}}\text{d}u\right| &\leq \exp\left( \frac{\log N}{W\log t}\right) \left( \frac{2|s|}{12N^2} + \frac{|s^2|}{12N^2} \right) \\
    &\leq \exp\left( \frac{\log N}{W\log t}\right) \left( \frac{t+\sigma_1}{6N^2} + \frac{t^2+\sigma_1^2}{12N^2} \right).
\end{align*}

We thus have that $|\zeta(s)|$ is upper bounded by
\begin{align*}
    \exp\left( \frac{\log N}{W\log t}\right) \left(\log N + \gamma + \frac{1}{N} +\frac{1}{64N^4} + \frac{1}{t} + \frac{t^2+\sigma_1^2+2(t+\sigma_1)-1}{12N^2} \right).
\end{align*}
It remains to choose $N$. We can minimise the expression by taking $N = \lfloor \eta t \rfloor + 1$ for some constant $\eta\in[2/t_0,1-1/t_0]$. This will make
\begin{equation*}
    \exp\left( \frac{\log N}{W\log t}\right) \leq \exp\left(\frac{1}{W} \right)
\end{equation*}
for all $t\geq t_0$.
The bound then becomes
\begin{align*}
    |\zeta(s)| &\leq e^{1/W} \left( \log t + \log(\eta+t^{-1}) + \gamma + \frac{1}{t} \left(\frac{1}{\eta} + 1\right) + \frac{1}{64\eta^2 t^2} \right. \\
    &\hspace{5cm} +\left. \frac{1}{12\eta^2} + \frac{1}{6\eta^2 t} + \frac{\sigma_1^2+2\sigma_1-1}{12\eta^2t^2} \right) \\
    &\leq e^{1/W} \log t + C_0(W,\sigma_1,t_0,\eta)
\end{align*}
for $t\geq t_0$.
\end{proof}

We also provide a corollary, specialised for $\sigma \ge 1$, which sharpens the above estimates.
\begin{corollary}\label{propA1}
Let $s=\sigma+ikt$ with $1 \le \sigma \le \sigma_1 \le 2$, $t\ge t_0\ge 3$, and a fixed $k\ge 1$. For $0< \eta/k <t_0$, we have
\begin{equation*}
|\zeta(\sigma +ikt)| < \log (kt) + C(\sigma_1,t_0,k,\eta),
\end{equation*}
where
\begin{align}\label{CA1}
C(\sigma_1,t_0,k,\eta) = &\log \left( \frac{1}{\eta}+\frac{1}{kt_0}\right) + \gamma + \frac{\eta^4}{64(kt_0)^4} \\
&+\frac{12+(c_1(\sigma_1,t_0,k)-(kt_0)^{-1})\eta^2}{12kt_0} +\frac{c_2(\sigma_1,t_0,k)\eta^2}{24}, \nonumber
\end{align}
with $\gamma$ denoting Euler's constant, and 
\begin{align*}
c_1(\sigma_1,t_0,k) &= \left( \frac{\sigma_1^2}{(kt_0)^2} +1\right)^{1/2},\\
c_2(\sigma_1,t_0,k) &= c_1(\sigma_1,t_0,k)\left( \frac{(\sigma_1+1)^2}{(kt_0)^2} +1\right)^{1/2}.
\end{align*}
\end{corollary}
\begin{proof}
The proof follows the argument of \cite[Proposition~A.1]{carneiro2022optimality}, originally due to Backlund \cite{backlund1916nullstellen}, and extends their result to $t< 500$. Note that here we have instead the choice of $N=\lfloor kt/\eta\rfloor +1$. A slightly better estimate of the harmonic sum from \cite[Lemma~2.1]{francis2021investigation} was also used. We omit the proof since the process is almost identical to that of Lemma \ref{zeta_bound}.
\end{proof}

To use Lemma \ref{zeta_bound}, we want to choose a $t_0$ and optimise over $\eta$. For the smallest $t_0 = 3$, we can only use $\eta=2/3$. However, as $t_0\rightarrow \infty$, the optimal $\eta$ quickly approaches $1/\sqrt{6}\approx 0.4082$. We also have the following corollary, which is not required for our purposes, but is useful to have for other applications. 
\begin{corollary}\label{cor:zeta-bound}
    In the region 
    \begin{equation}\label{sig reg 3}
1-\frac{1}{W_0 \log t} \le \sigma \leq 1,
\end{equation}
we have for $t\geq t_0\geq 3$, 
    \begin{equation}
        |\zeta(s)| \le  1.1971\log t + B_1,
    \end{equation}
    where Table \ref{table:1} gives values of $B_1$ for select $t_0$.
\end{corollary}
\begin{proof}
For a fixed $t_0$ and $\sigma_1 = 1$, in Lemma \ref{zeta_bound} optimise over $\eta$ to find the minimal $B_1$.
\end{proof}
\begin{table}[ht!]
\centering
\begin{tabular}{llll}
$t_0$  & $B_1$  &  $\eta$ \\ \hline
$3$    & $2.1173$ &   $2/3$  \\
$10^3$ & $0.2254$  &  $0.41$ \\
$10^{12}$ & $0.2171$  &  $0.41$
\end{tabular}
\caption{Values of $B_1$ for Corollary \ref{cor:zeta-bound} with the corresponding optimised $\eta$ from Lemma \ref{zeta_bound}, where $\sigma_1 =1$.}
\label{table:1}
\end{table}

Corollary \ref{cor:zeta-bound} can be used in conjunction with other estimates for $\zeta(s)$ in the critical strip, such as Ford's Theorem 1 in \cite{ford_2002_zeta-bounds} or Trudgian's Lemma 2.1 \cite{trudgian2016improvements2}. Ford's bound, while asymptotically better, is only sharper than bounds of the form $\zeta(s)\ll \log t$ for very large $t$. For $\sigma\in[1-1/(W_0 \log t), 1]$ and $3\leq t\leq 10^{10^5}$, Corollary \ref{cor:zeta-bound} is an improvement to \cite[Theorem~1]{ford_2002_zeta-bounds}. This makes it useful for results over smaller $t$.

It is not uncommon to degrade asymptotically sharper results for $\zeta(s)$ to $|\zeta(s)|\leq C\log t$ for the sake of a smaller main-term constant $C$. Corollary \ref{cor:zeta-bound} now directly offers a better constant. For instance, over $\sigma\in[1-1/(W_0 \log t), 1]$ and $t\ge 3$ our result implies
\begin{equation*}
    |\zeta(s)|\leq 3.13\log t,
\end{equation*}
whereas degrading a Richert-type bound like Ford's \cite[Theorem~1]{ford_2002_zeta-bounds} would give
\begin{equation*}
    |\zeta(s)|\leq 102\log t,
\end{equation*}
and a bound obtained via the Phragm\'{e}n--Lindel\"{o}f principle like (2.4) in \cite[Lem.~2.1]{trudgian2016improvements2} using estimates on the $1/2$-line (see Hiary, Patel, and Yang \cite{hiarypatelyang2022}) and estimates on the $1$-line (see Backlund \cite{backlund1916nullstellen}) would yield
\begin{equation*}
    |\zeta(s)|\leq 364\log t.
\end{equation*}

We also remark that although Lemma \ref{zeta_log_bound 2/3} is asymptotically superior to Corollary \ref{propA1}, one would expect that the former beats the latter only for a sufficiently large $t$; approximately $t\ge \exp(58^3)\ge 10^{84736}$, say. This can be improved by lowering the constant in Lemma \ref{1 line 2/3 bd}, and correspondingly $C_3$.

Next we have some bounds on $\zeta(\sigma)$ for real $\sigma >1$. These are useful in their own right, but we show they can be used to obtain good bounds on $-\zeta'(\sigma)/\zeta(\sigma)$ in a similar region.

\begin{lemma}\cite[Lemma~5.4]{ramare2016explicit}\label{bastien bound}
Let $\gamma$ denote Euler's constant. For $\sigma >1$ we have
\begin{equation*}
|\zeta(\sigma + it)| \le \zeta(\sigma) \le \frac{e^{\gamma(\sigma - 1)}}{\sigma - 1}.
\end{equation*}
\end{lemma}
\begin{lemma}\label{zeta_real_lower_bound_0}
Let $\gamma_n$ be the Laurent--Stieltjes constants (with $\gamma := \gamma_0$). For integers $k$ such that $1 \le k \le 1500$, and real $\sigma$ satisfying $1< \sigma \le \sigma_k<3$, we have
\begin{equation*}
\zeta(\sigma) \ge \frac{1}{\sigma-1} +\gamma +\phi_0(\sigma,k),
\end{equation*}
where
\begin{equation}\label{phi0 >0}
\phi_0(\sigma,k) = \sum_{n=1}^k \frac{(-1)^n(\sigma-1)^n}{n!}\gamma_n + \frac{1}{(\sigma-3)}\left(\frac{\sigma-1}{2}\right)^{k+1} \ge 0.
\end{equation}
In addition, we note some admissible pairs:
\begin{equation*}
(\sigma_k,k) = (1.45,1),\,(2,3),\,(2.51,10),\,(2.97,500),\,(2.98,1000),\,(2.99,1500).
\end{equation*}
\end{lemma}

\begin{proof}
To prove the first assertion of the lemma, we have from the Laurent series expansion (about $\sigma=1$) of the zeta function 
\begin{align}\label{zeta_laurent_series}
\zeta(\sigma) &= \frac{1}{\sigma-1}+\gamma+\sum_{n=1}^k\frac{(-1)^n(\sigma-1)^n\gamma_n}{n!}+\sum_{n=k+1}^\infty\frac{(-1)^n(\sigma-1)^n\gamma_n}{n!} \\
&\ge \frac{1}{\sigma-1}+\gamma+\sum_{n=1}^k\frac{(-1)^n(\sigma-1)^n\gamma_n}{n!}-\left|\sum_{n=k+1}^\infty\frac{(-1)^n(\sigma-1)^n\gamma_n}{n!}\right|.\nonumber
\end{align}
We have the bounds on $\gamma_n$ due to Lavrik \cite{lavrik1976principal}:
\begin{equation}\label{lavrik bounds}
|\gamma_n| \le \frac{n!}{2^{n+1}}, \quad \quad n=1,2,3,\ldots,
\end{equation}
hence
\begin{equation*}
\left|\sum_{n=k+1}^\infty\frac{(-1)^n(\sigma-1)^n\gamma_n}{n!}\right| \le \sum_{n=k+1}^\infty \frac{(\sigma-1)^n}{2^{n+1}} = -\frac{1}{(\sigma-3)}\left(\frac{\sigma-1}{2}\right)^{k+1}
\end{equation*}
by the geometric series with $|1-\sigma|<2$. Taking the lower bound from this, one can then verify that $\phi_0(\sigma,k) \ge 0$ for $1<s\le \sigma_k$ by calculating the first $k$ terms. 
\end{proof}
It is likely that $\phi_0(\sigma,k)$ is always positive for any $\sigma >1$, although proving that would require a different approach, perhaps one akin to \cite[Lemma~5.3]{ramare2016explicit}. Our method in Lemmas \ref{zeta_real_lower_bound_0} and \ref{zeta'_zeta_real_upper_bound}, while simpler, is restricted by bounds on $|\gamma_n|$ which loses some information about sign changes. It is worth noting that bounds sharper than \eqref{lavrik bounds} are available, but are more complicated to work with. 
\begin{lemma}\label{zeta'_zeta_real_upper_bound}
Let the hypotheses of Lemma \ref{zeta_real_lower_bound_0} be fulfilled. Then for $\sigma$ real, we have
\begin{equation*}
-\zeta'(\sigma) \le \frac{1}{(\sigma-1)^2}+\frac{1}{(\sigma-3)^2}, \qquad 1<\sigma< 3,
\end{equation*}
and
\begin{equation*}
\left| \frac{\zeta'(\sigma+it)}{\zeta(\sigma+it)}\right| \le \frac{-\zeta'(\sigma)}{\zeta(\sigma)} \le \frac{\phi_1(\sigma,k)}{\sigma-1}, \qquad 1<\sigma\le \sigma_k,
\end{equation*}
where
\begin{equation}\label{phi1}
\phi_1(\sigma,k) = \frac{(\sigma-3)^2 +(\sigma-1)^2}{(\sigma-3)^2(1+(\sigma-1)(\phi_0(\sigma,k)+\gamma))}.
\end{equation}
\end{lemma}
\begin{proof}
The first inequality of the second assertion is trivial. The subsequent inequality is proved by differentiating \eqref{zeta_laurent_series} and repeating a process similar to the proof of Lemma \ref{zeta_real_lower_bound_0} (not splitting the sum for brevity), to obtain
\begin{align*}
-\zeta'(\sigma) &\le \frac{1}{(\sigma-1)^2}+\sum_{n=1}^\infty \frac{n(\sigma-1)^{n-1}}{2^{n+1}} \\
&= \frac{1}{(\sigma-1)^2}+\frac{1}{2(\sigma-1)}\sum_{n=1}^\infty n\left(\frac{\sigma-1}{2}\right)^n =  \frac{1}{(\sigma-1)^2}+\frac{1}{(\sigma-3)^2}
\end{align*}
for $|1-\sigma|<2$ by identification with the polylogarithm $\textup{Li}_{-1}(z)=z/(1-z)^2$. Dividing by the bound in Lemma \ref{zeta_real_lower_bound_0} proves the result.
\end{proof}
These bounds are good for $\sigma$ close to $1$. For instance, we have that 
\begin{equation}\label{compare delange}
    \phi_1(\sigma,3)<1 \qquad \text{for}\qquad \sigma \le 1.83,
\end{equation}
which is an improvement on Delange's result \cite{delange1987remarque} for a fixed $\sigma \le 1.65$. Although the bounds get sharper as $k$ increases, in practice, taking $k\le 10$ is more than sufficient for small $\sigma$ and any savings gained by subsequent increases in $k$ would be after the fifth decimal place.

\begin{corollary}\label{zeta'_zeta_real_upper_bound cor}
Let $\phi_1(\sigma,k)$ be defined as in \eqref{phi1}. For $1< \sigma_0 \le \sigma\le 1.5$, we have
\begin{equation*}
\left| \frac{\zeta'(\sigma+it)}{\zeta(\sigma+it)}\right| \le \frac{-\zeta'(\sigma)}{\zeta(\sigma)} \le \frac{\phi_2(\sigma_0)}{\sigma-1},
\end{equation*}
where $\phi_2(\sigma_0)=\max\left\lbrace \phi_1(\sigma_0,10),\,0.852 \right\rbrace$.
\end{corollary}
\begin{proof}
For $k=10$, we have from Lemma \ref{zeta_laurent_series} that $\phi_0(\sigma,10) \ge 0$ in our stated range of $\sigma$. It can also be verified that $\phi_1(\sigma,10)$ is decreasing in the region $1<\sigma \le 1.48$, and therefore $\phi_1(\sigma,10)\le \phi_1(\sigma_0,10)$.

To complete the proof, verify that $\phi_1(\sigma,10)\le 0.852$ in the region $1.48\le \sigma \le 1.5$.
\end{proof}

Finally, we have a lower bound on $|\zeta(s)|$, for $\sigma >1$, which is better than the trivial bound $|\zeta(s)|\ge \zeta(2\sigma)/\zeta(\sigma)$ over certain ranges of $\sigma$ (for more information, see \cite{leongmossinghoffArxiv}). This is the $\log t$ analogue of \cite[Lemma~8]{cully2024explicit}.

\begin{lemma}\label{inverse zeta trig poly}
Let 
$\kappa, \, \sigma_1$ be positive parameters, and $\gamma$ be Euler's constant. Then for $t\ge t_0\ge 3$ and $1+\kappa \le \sigma_1 \le 2$, we have
\begin{equation*}
    \left| \frac{1}{\zeta(1+\kappa +it)}\right| \le  V(\kappa,\sigma_1,t_0,\eta) \log t
\end{equation*}
with
\begin{align*}\label{V}
V(\kappa,\sigma_1,t_0,\eta) := \left(\frac{\exp\left( \gamma \kappa\right)}{\kappa\log t_0}\right)^{3/4}\left(1+ \frac{\log 2 + C(\sigma_1, t_0, 2, \eta)}{\log t_0}\right)^{1/4}, 
\end{align*}
where $\eta$ and $C(\sigma_1,t_0,k,\eta)$ are defined in Corollary \ref{propA1}.
\end{lemma}
\begin{proof}
Since $\sigma >1$, we use the classical non-negativity argument involving the trigonometric polynomial $2(1+\cos\theta)^2=3 + 4\cos\theta + \cos2\theta$ (see \cite[\S~3.3]{titchmarsh1986theory} and \cite{leongmossinghoffArxiv}). That is, we use $\zeta^3(\sigma)|\zeta^4(\sigma+it)\zeta(\sigma +2it)| \ge 1$, which implies
\begin{equation}\label{trig ineq}
    \left| \frac{1}{\zeta(\sigma+it)}\right| \le |\zeta(\sigma)|^{3/4}|\zeta(\sigma + 2it)|^{1/4}.
\end{equation}

Taking $\sigma = 1+\kappa$ in \eqref{trig ineq} gives
\begin{equation}\label{trig poly ineq}
    \left| \frac{1}{\zeta(1 + \kappa +it)}\right| \le |\zeta(1 + \kappa)|^{3/4}|\zeta(1 + \kappa + 2it)|^{1/4}.
\end{equation}
The two factors on the right-hand side can be bounded with Lemma \ref{bastien bound} and Corollary \ref{propA1}, respectively, with the aid of the bound 
\begin{align*}
    \log 2t &\le \left( 1 + \frac{\log 2}{\log t_0} \right) \log t. 
\end{align*}
This proves the lemma.
\end{proof}

\section{Estimates for $|\zeta'(s) / \zeta(s)|$}\label{section zeta'/zeta}

\subsection{Assuming a partial Riemann hypothesis}\label{partial RH}

In this section, we assume $RH(T)$, and estimate $|\zeta'(s)/\zeta(s)|$ in a right-half plane: $\sigma > 1/2 +\epsilon$, $t_0\le t\le T$.

\begin{theorem}\label{thm zeta'/zeta RH}
Assume $RH(T)$. Let $\alpha_0$ and $\epsilon$ be parameters, such that $0< \epsilon \le 1/2$ and $0<\alpha_0 <1$.
In the region
\begin{equation*}
    \sigma \ge 1+\epsilon-\frac{\alpha_0}{2} \quad \text{and}\quad 3 \le t_0\le t\le T-\frac{1}{2},
\end{equation*}
we have 
\begin{equation*}
\left|\frac{\zeta'(\sigma+it)}{\zeta(\sigma+it)}\right|\le Q_{RH(T)} \log t,
\end{equation*}
with 
\begin{align}\label{QRH(T) const}
Q_{RH(T)} &= \max\Bigg\lbrace \frac{4}{(1-\alpha_0)^2}\Bigg( \frac{1}{6}+2B_{t_0}+\frac{\log A_3}{\log t_0}\Bigg)\, , \, \frac{1}{(\epsilon +(\alpha_0/2))\log t_0}\Bigg\rbrace, 
\end{align}
where $A_3$ is defined in \eqref{A3 RH}, and
\begin{equation}\label{Bt0}
B_{t_0} = \begin{cases}
\frac{1}{e}  &\text{ if } t_0 \le e^e, \\
\frac{\log\log t_0}{\log t_0}    &\text{ if } t_0 > e^e,
\end{cases}
\end{equation}
provided that 
\begin{equation}\label{tcond RH1}
    t^{1/6}\log t \ge \frac{\zeta(\frac{3}{2}+\epsilon)}{0.618 (1+a_1 (\frac{3}{2}+\epsilon,1.31,t_0))^{1/6}}
\end{equation}
is satisfied, with $a_1(\sigma,1.31,t)$ defined as in Lemma \ref{plpbound}. 

Furthermore, by assuming $RH$ instead of $RH(T)$, the upper bound restriction on $t$ can be removed. 
\end{theorem}

\begin{proof}
We will first prove the result for $T_0$ and then make the substitution to $t$ at the end. The idea is to construct discs $|s-s_0|\le 1/2$ which have origins slightly to the right of the $1$-line, centred at $s_0=\sigma_0+iT_0$, where 
\begin{equation*}
\sigma_0 = 1+ \epsilon,\qquad e+r\le t_0\le T_0 \le T-1/2,
\end{equation*}
and $\epsilon \le 1/2$ is a parameter to be determined later. We then extend their radii into the critical strip, so that Corollary \ref{cor_3.1_HB_F-M_RH} applied with
\begin{equation*}
     f(s):= \zeta(s),\qquad  s_0=\sigma_0+iT_0, \qquad r=1/2, 
\end{equation*}
gives us the desired estimates. Due to $RH(T)$, $\zeta(s)$ is analytic and non-zero in the discs.

In order to do this, we first need to determine $A_1$. Let $V(\epsilon,\sigma_1,t_0,\eta)$ be defined as in Lemma \ref{inverse zeta trig poly}, with $\sigma_1$ and $\eta$ as parameters. Then by Corollary \ref{plp_cor} with $\delta_r = r+\epsilon$, and Lemma \ref{inverse zeta trig poly} with $\kappa  = \epsilon$, we have for $1/2\le \sigma \le \sigma_0+r$ and each $T_1 \in [T_0-r,T_0+r]$,
\begin{equation}\label{zeta_upbounddone0}
    \left|\frac{\zeta(\sigma+iT_1)}{\zeta(\sigma_0+iT_0)}\right| \le A_1 := A_3 T_0^{1/6}(\log T_0)^2
\end{equation}
where
\begin{align}\label{A3 RH}
    A_3 := 0.618 &(1+a_0(\sigma_0+r,1.31,t_0))(1+a_1(\sigma_0+r,1.31,t_0))^{7/6}\nonumber \\
    &\boldsymbol{\cdot}\left(1+\frac{r}{t_0} \right)^{1/6} \left(1+\frac{\log \left(1+\frac{r}{t_0}\right)}{\log t_0}\right)V(\epsilon,\sigma_1,t_0,\eta);
\end{align}
provided \eqref{tcond} is satisfied. Clearly, the upper bound in \eqref{zeta_upbounddone0} will also apply to $|\zeta(\sigma+iT_0)/\zeta(\sigma_0+iT_0)|$.

Now choosing a fixed positive $\alpha_0< 1$, we have by Corollary \ref{cor_3.1_HB_F-M_RH},
\begin{align}\label{compare with tenen triv}
    \left| \frac{\zeta'(\sigma+iT_0)}{\zeta(\sigma+iT_0)}\right| &\le \frac{4}{(1-\alpha_0)^2}\left( \frac{1}{6}\log T_0 +2\log\log T_0 +\log A_3\right) \nonumber \\
    &\le \frac{4}{(1-\alpha_0)^2}\left( \frac{1}{6} +2B_{t_0} +\frac{\log A_3}{\log t_0}\right)\log T_0,
\end{align}
which holds for
\begin{equation*}
    1+\epsilon -\frac{\alpha_0}{2}\le \sigma \le 1+\epsilon+\frac{\alpha_0}{2}.
\end{equation*}
Note that $B_{t_0}$ is defined in \eqref{Bt0}, and is utilised since $\log\log t/\log t$ attains its maximum at $t=e^e$, after which it is decreasing.

To obtain a bound that is valid also for $\sigma\ge 1+\epsilon+\alpha_0 /2$, we use the trivial bound for all $\sigma >1$ \cite[Lemma~70.1]{tenenbaum1988rrhall}:
\begin{equation}\label{tenen trivial}
\left| \frac{\zeta'(s)}{\zeta(s)}\right| \le -\frac{\zeta'(\sigma)}{\zeta(\sigma)} < \frac{1}{\sigma -1}
\end{equation}
with $\sigma =1+\epsilon+\alpha_0/2$, and then take the maximum of \eqref{compare with tenen triv} and \eqref{tenen trivial}.

Finally, making the substitution from $T_0$ to $t$ proves the result.
\end{proof}

Assuming $RH(H_0)$ allows us to prove Corollary \ref{cor zeta'/zeta RH}, which holds for $t\le H$.

\begin{proof}[Proof of Corollary \ref{cor zeta'/zeta RH}]
For $t\ge t_0$ and $\sigma\ge \sigma_0$, in Theorem \ref{thm zeta'/zeta RH} set $\sigma_0 = 1+\epsilon -\frac{\alpha_0}{2}$. Then fixing $t_0$ and $\sigma_0$, optimise \eqref{QRH(T) const} over $\epsilon$ (which also determines $\alpha_0$), and also over $\sigma_1, \, \eta$ (due to the quantity $V$ in $A_3$), while ensuring that the restrictions set in Theorem \ref{thm zeta'/zeta RH} (and by extension Lemma \ref{inverse zeta trig poly} and Corollary \ref{propA1}), hold. The value $t_0$ is chosen by finding the smallest value of $t$ that satisfies \eqref{tcond RH1}.
\end{proof}
\begin{table}[ht!]
\centering
\setlength\tabcolsep{12pt}
\begin{tabular}{|c|c|c||c|c|c|} 
 \hline
 $\sigma_0$ & $t_0$& $Q_{RH(H_0)}$ & $\epsilon$  & $\sigma_1$ & $\eta$ \\ \hline
 $0.51$ & $15$ & $31447$ & $0.000462$  & $1.965754$ & $2.608856$\\
 $0.55$ & $14$ & $1086$ & $0.002766$  & $1.901324$ & $3.329589$\\
 $0.6$ & $14$ & $249.260$ & $0.006041$  & $1.795116$ & $3.162609$\\
 $0.65$ & $14$ & $104.992$ & $0.009573$  & $1.846979$ & $3.195987$\\
 $0.7$ & $14$ & $56.739$ & $0.013287$  & $1.797626$ & $3.198866$\\
 $0.75$ & $14$ & $35.162$ & $0.017144$  & $1.155376$ & $2.868261$\\
 $0.8$ & $14$ & $23.759$ & $0.021126$  & $1.392644$ & $3.173843$\\
 $0.85$ & $14$ & $17.050$ & $0.025202$  & $1.763188$ & $3.402204$\\
 $0.9$ & $13$ & $13.123$ & $0.029430$ & $1.227811$ & $3.283187$\\
 $0.95$ & $13$ & $10.175$ & $0.033682$  & $1.227876$ & $3.283802$\\
 $1$ & $13$ & $8.101$ & $0.037999$  & $1.889284$ & $3.054339$\\
 \hline
\end{tabular}
\caption{Values for $Q_{RH(H_0)}$  and $\sigma\ge \sigma_0$ in Corollary \ref{cor zeta'/zeta RH}, with corresponding $\epsilon,\, \sigma_1, \, \eta$ (all rounded to six decimal places).
Each entry is valid for $t_0\le t\le H$.}
\label{table:QRH}
\end{table}

We note that \cite[Theorem 1.1]{chirre2024bounding} implies that
\begin{equation*}
    \left| \frac{\zeta'(1+it)}{\zeta(1+it)}\right| \le 0.639\log t
\end{equation*}
for $10^6 \le t\le 2.99997\cdot 10^{12}$. This result allows us to show that
\begin{equation}\label{wascor4}
    \left| \frac{\zeta'(1+it)}{\zeta(1+it)}\right| \le 40.440\log\log t
\end{equation}
holds unconditionally for $e^e \le t\le 2.99997\cdot 10^{12}$.

To prove this, recall from \S\ref{properties} that $RH$ has been verified up to $3.000175332800 \cdot 10^{12}$. Now note that $\log t/\log\log t$ is increasing for $t\ge e^e$. Then repeating the process in the proof of Corollary \ref{cor zeta'/zeta RH}, for $e^e \le t\le 10^6$, we obtain
\begin{equation*}
    \left| \frac{\zeta'(1+it)}{\zeta(1+it)}\right| \le 7.686\log t \le 40.440\log\log t,
\end{equation*}
with rounded parameters $(\epsilon,\sigma_1,\eta)=(0.037816,1.395842,3.177141)$. The second inequality above uses the fact that $t$ attains its maximum at $10^6$.

We remark that \eqref{wascor4} is an unconditional result due to the partial verification of $RH$. While this is nowhere near as sharp as the bounds in \cite{chirre2024bounding}, it nevertheless holds for a wider range and might be useful for smaller values of $t$.

If desired, one can extend the range for which the corollary is unconditional beyond this (albeit still in a finite range $t_0\le t\le T$), at the expense of a worse leading constant. This will prove a poor man's version of \cite[Theorem 1.1]{chirre2024bounding}, but with a simpler method. As mentioned in \S\ref{intro}, the range of validity is wider but the constants are worse. This extension requires a modification in the proof of Theorem \ref{thm zeta'/zeta RH}, by fixing $r=\epsilon+ 1/W\log T$, utilising a classical zero-free region, and also using Lemma \ref{zeta_bound} in place of Corollary \ref{plp_cor} when estimating $A_1$ in \eqref{zeta_upbounddone0}. 

\subsection{Above the Riemann height}\label{above H}

\begin{table}[ht!]
\centering
\setlength\tabcolsep{12pt}
\renewcommand{\arraystretch}{1}
\rule{0pt}{4ex}
\begin{tabular}{|c|c||c|c|c|c|} 
 \hline
 $W$ & $Q$ & $\beta$ & $\epsilon_1$ & $\sigma_1$ & $\eta$ \\ \hline
 $5.559$ & $928465$  & $ 0.999980 $ & $ 0.012554 $ & $ 1.756408 $ & $ 1.050104 $\\
 $5.56$ & $170199$  & $ 0.999890 $ & $ 0.012551 $ & $ 1.937703 $ & $ 2.228745 $\\
 $5.5601$ & $157356$  & $ 0.999881 $ & $ 0.012563 $ & $ 1.393177 $ & $ 4.286069 $ \\
 $5.5602$ & $146346$  & $ 0.999872 $ & $ 0.012587 $ & $ 1.710791 $ & $ 1.367807 $ \\ 
 $5.5603$ & $136750$  & $ 0.999863 $ & $ 0.012582 $ & $ 1.776753 $ & $ 1.141731 $ \\
 $5.5604$ & $128379$  & $ 0.999854 $ & $ 0.012712 $ & $ 1.889117 $ & $ 2.700506 $\\
 $5.5605$ & $120884$  & $ 0.999845 $ & $ 0.012614 $ & $ 1.711461 $ & $ 3.302948 $\\
 $5.561$ & $93736$ & $ 0.999800 $ & $ 0.012690 $ & $ 1.339371 $ & $ 2.627981 $\\
 $5.562$ & $64686$ & $ 0.999710 $ & $ 0.012842 $ & $ 1.299159 $ & $ 2.599102 $ \\
 $5.563$ & $49364$ & $ 0.999620 $ & $ 0.012699 $ & $ 1.427545 $ & $ 1.952944 $ \\
 $5.564$ & $39919$ & $ 0.999531 $ & $ 0.012747 $ & $ 1.557141 $ & $ 2.691489 $\\
 $5.565$ & $33530$ & $ 0.999441 $ & $ 0.013347 $ & $ 1.715049 $ & $ 4.337899 $\\
 $5.566$ & $28887$ & $ 0.999351 $ & $ 0.013295 $ & $ 1.234376 $ & $ 4.393614 $\\
 $5.567$ & $25383$ & $ 0.999262 $ & $ 0.013634 $ & $ 1.973530 $ & $ 3.744506 $ \\
 $5.568$ & $22625$ & $ 0.999172 $ & $ 0.013225 $ & $ 1.814250 $ & $ 2.845626 $\\
 $5.569$ & $20414$ & $ 0.999082 $ & $ 0.013180 $ & $ 1.641377 $ & $ 1.985693 $\\
 $5.57$ & $18592$ & $ 0.998993 $ & $ 0.012868 $ & $ 1.322806 $ & $ 4.407476 $\\ 
 $5.575$ & $12870$ & $ 0.998545 $ & $ 0.012963 $ & $ 1.920818 $ & $ 3.555767 $ \\
 $5.58$ & $9863$ & $ 0.998102 $ & $ 0.019214 $ & $ 1.690970 $ & $ 2.264745 $ \\
 $5.59$ & $6698$ & $ 0.997207 $ & $ 0.012557 $ & $ 1.929368 $ & $ 4.355523 $\\
 $5.6$ & $5082$ & $ 0.996321 $ & $ 0.015401 $ & $ 1.421225 $ & $ 3.490106 $ \\
 $5.7$ & $1501$ & $ 0.987620 $ & $ 0.025973 $ & $ 1.590960 $ & $ 3.434482 $\\
 $5.8$ & $888.269$ & $ 0.979209 $ & $ 0.018833 $ & $ 1.305491 $ & $ 4.203009 $ \\ 
 $5.9$ & $635.099$ & $ 0.971088 $ & $ 0.021934 $ & $ 1.977612 $ & $ 4.803605 $\\
 $6$ & $496.670$ & $ 0.963240 $ & $ 0.027122 $ & $ 1.420679 $ & $ 3.567866 $\\
 $6.1$ & $409.353$ & $ 0.955640 $ & $ 0.016196 $ & $ 1.713398 $ & $ 1.710689 $ \\
 $6.2$ & $349.287$ & $ 0.948293 $ & $ 0.019949 $ & $ 1.623774 $ & $ 3.117029 $\\
 $6.3$ & $305.581$ & $ 0.941204 $ & $ 0.066228 $ & $ 1.701866 $ & $ 3.427829 $\\
 $6.4$ & $272.026$ & $ 0.934286 $ & $ 0.023358 $ & $ 1.163604 $ & $ 3.125341 $ \\
 $6.5$ & $246.170$ & $ 0.927735 $ & $ 0.255900 $ & $ 1.454736 $ & $ 3.093587 $ \\
 $6.75$ & $199.284$ & $ 0.911787 $ & $ 0.057471 $ & $ 1.361313 $ & $ 4.269028 $\\
 $7$ & $168.924$ & $ 0.897081 $ & $ 0.059017 $ & $ 1.777110 $ & $3.321819 $\\
 $8$ & $109.668$ & $ 0.847438 $ & $ 0.038048 $ & $ 1.707618 $ & $ 2.623961 $\\
 $9$ & $84.858$ & $ 0.808831 $ & $ 0.031281 $ & $ 1.249275 $ & $ 2.976333 $\\
 $10$ & $71.220$ & $ 0.777942 $ & $ 0.016334 $ & $ 1.624690 $ & $ 4.127955 $ \\
 $11$ & $62.611$ & $ 0.752706 $ & $ 0.084730 $ & $ 1.177847 $ & $ 1.792189 $ \\ 
 $12$ & $56.653$ & $ 0.731621 $ & $ 0.020106 $ & $ 1.606738 $ & $ 2.688990 $ \\
 $13$ & $52.306$ & $ 0.713814 $ & $ 0.041793 $ & $ 1.671118 $ & $ 3.367414 $ \\
 \hline
\end{tabular}
\caption{Values for $Q$ in Corollary \ref{cor main zeta'/zeta}, with corresponding $W$ and parameters $\beta$, $\epsilon_1$, $\sigma_1$, $\eta$ (all rounded to six decimal places). Each entry is valid for $t\ge 13$.}
\label{table:Q}
\end{table}

\begin{theorem}\label{main zeta'/zeta for H}
Let $d\le 1/W_0$, $\epsilon_1 \le 1/2$, and $\beta <1$ be positive real parameters, subject to \eqref{alpha cond0} and \eqref{beta conds}, and such that $W>W_0$, with $W$ defined as 
\begin{equation}\label{W define}
    W := \left(d\beta (1+a_{\epsilon_1}(t_0)^{-1})  -d\right)^{-1},
\end{equation}
where the function $a_{\epsilon_1}(t_0)$ also depends on $\epsilon_1$ and is defined in \eqref{a2(t0)}. 

Then in a region
\begin{equation*}
    \sigma \ge 1-\frac{1}{W\log t}, \qquad t\ge t_0\ge H
\end{equation*}
where $\zeta(\sigma+it)\neq 0$, we have 
\begin{equation*}
\left|\frac{\zeta'(\sigma+it)}{\zeta(\sigma+it)}\right|\le Q_H \log t,
\end{equation*}
where 
\begin{equation*}\label{Q1}
Q_H = \max\left\lbrace \lambda_1 \left( \frac{1}{6}+\frac{2\log\log t_0}{\log t_0}+\frac{\log A_3}{\log t_0}\right)+\lambda_2,\, \frac{1}{(d\beta (1+a_{\epsilon_1}(t_0)^{-1}) +d)}\right\rbrace,
\end{equation*}
with
\begin{equation}\label{lambda}
\lambda_1 = \frac{16\beta}{(1-\beta)}\left(1-\frac{8d}{2d+\log t_0}\right)^{-2} \qquad \text{and}\qquad \lambda_2 = \frac{1+\beta}{d(1-\beta)},
\end{equation}
provided that 
\begin{equation}\label{tcond not RH1}
    t^{1/6}\log t \ge \frac{\zeta(\frac{3}{2}+\frac{2d}{\log t_0})}{0.618 (1+a_1 (\frac{3}{2}+\frac{2d}{\log t_0},1.31,t_0))^{1/6}}
\end{equation}
is satisfied. 

The quantity $A_3$ is defined in \eqref{A3}, and depends on $d$ and $t_0$, on $a_0$ and $a_1$, which are defined as in Lemma \ref{plpbound}, and on $V(\kappa,\sigma_1,t_0,\eta)$ and its corresponding parameters, defined as in Lemma \ref{inverse zeta trig poly} and subject to \eqref{s1 eta cond}. 

\end{theorem}

\begin{proof}
The proof is similar to that of Theorem \ref{thm zeta'/zeta RH}. Suppose $t_0\ge H$ and let there be a positive real parameter $d\le 1/W_0$ to be determined later. As before, we construct discs which have origins slightly to the right of the $1$-line, at $s_0 =\sigma_0 +iT_0$, where $T_0 \ge t_0$,  
\begin{equation}\label{sigma0}
\sigma_0 
= 1+ \frac{d}{\log T_0}, \qquad r= \frac{1}{2}+\frac{d}{\log T_0},
\end{equation}
extend their radii slightly into the critical strip, and then apply Lemma \ref{lem_3.2_HB} with $f(s) = \zeta(s)$, $s=\sigma+iT_0$, and $A_1$ defined by \eqref{zeta_upbounddone0}. This time however, we do not have the assumption of $RH(T)$ to help us with potential zeros. We will instead utilise the zero-free region mentioned in \S\ref{properties} and keep our discs within said region. This requires discs with radii decreasing to $0$ at the rate of $1/\log T_0$, similar to the zero-free region.

Let $\epsilon_1$ be a fixed positive real parameter, and for the sake of brevity define the decreasing function (for $T_0 >1$)
\begin{equation}\label{a2(t0)}
    a_{\epsilon_1}(T_0) : = 1+ \frac{\log\left( 1+\frac{\epsilon_1}{\log T_0}\right)}{\log T_0} > 1.
\end{equation} 
To ensure that the non-vanishing condition on $f(s)=\zeta(s)$ is met, notice that if we impose the condition $\alpha r \le  \epsilon_1 \le  r$, then the disc $|s-s_0|\le  \alpha r$ with
\begin{equation*}
    \alpha r\le  \sigma_0- \left( 1- \frac{d}{\log (T_0 +\epsilon_1)}\right),
\end{equation*}
is completely contained in the zero-free region as defined in \S\ref{properties}. This also implies that the following choice of $\alpha$ and bounds are admissible:
\begin{equation}\label{alpha cond0}
    \alpha := \frac{2d\left(1+a_{\epsilon_1}(t_0)^{-1}\right)}{\log T_0 +2d} < \frac{4d}{2d+\log t_0} \qquad \text{and} \qquad \frac{2d}{\log t_0}\le \epsilon_1 \le \frac{1}{2}, 
\end{equation}
which in turn requires that $d < (\log t_0 )/6$ since we need $\alpha < 1/2$ for Lemma \ref{lem_3.2_HB}. This is easily achieved since $d \le 1/W_0 < (\log t_0 )/6$ for all $t_0 \ge 3$.

Next, we determine $A_1$ and $A_2$. Suppose \eqref{tcond} is satisfied with $1+\delta_r = \sigma_0+r$ and define $A_1$ (which depends on $A_3$) by \eqref{zeta_upbounddone0}, but use \eqref{sigma0} with the argument of Lemma \ref{inverse zeta trig poly} to replace \eqref{A3 RH} by
\begin{align}\label{A3}
    &A_3 := 0.618 \left(1+a_0\left(\frac{3}{2}+\frac{2d}{\log t_0},1.31,t_0\right)\right)\left(1+a_1\left(\frac{3}{2}+\frac{2d}{\log t_0},1.31,t_0\right)\right)^{7/6}\nonumber \\
    &\boldsymbol{\cdot}\left(1+\frac{1}{2t_0}+\frac{d}{t_0\log t_0} \right)^{1/6} \left(1+\frac{\log \left(1+\frac{1}{2t_0}+\frac{d}{t_0\log t_0}\right)}{\log t_0}\right)V\left(\frac{d}{\log t_0},\sigma_1,t_0,\eta\right),
\end{align}
where $\sigma_1$ and $\eta$ satisfies
\begin{equation}\label{s1 eta cond}
   1+ \frac{d}{\log t_0}\le \sigma_1 \le 2 \qquad \text{and}\qquad 0<\eta < 2t_0.
\end{equation}

On the other hand, $A_2$ is determined by invoking \eqref{tenen trivial}, which gives us
\begin{equation*}
    \left|\frac{\zeta'(s_0)}{\zeta(s_0)} \right| \le A_2 := \frac{1}{d}\log T_0.
\end{equation*}

Now, for any results from Lemma \ref{lem_3.2_HB} to be meaningful for us, we require the left most point of $|s-s_0|\le \alpha\beta r$ to lie within the zero-free region and in the critical strip. In other words, $\sigma_0 -\alpha\beta r \le 1$, which implies
\begin{equation}\label{beta conds}
\beta \ge \frac{2d}{\alpha(2d+\log T_0)} 
\ge \left(1+\frac{1}{a_{\epsilon_1}(t_0)}\right)^{-1} \ge \frac{1}{2}.
\end{equation}
In particular, since
\begin{equation*}
    \sigma_0 -\alpha\beta r = 1- \left(d\beta (1+a_{\epsilon_1}(t_0)^{-1})  -d\right)\frac{1}{\log T_0},
\end{equation*}
to obtain a result that holds for
\begin{equation*}
   1- \left(d\beta (1+a_{\epsilon_1}(t_0)^{-1})  -d\right)\frac{1}{\log T_0} \le \sigma \le 1+ \left(d\beta (1+a_{\epsilon_1}(t_0)^{-1}) +d\right)\frac{1}{\log T_0},
\end{equation*}
we choose our parameters $d,\beta,\epsilon_1,t_0$ such that $W> W_0$, where we define
\begin{equation*}
    W := \left(d\beta (1+a_{\epsilon_1}(t_0)^{-1})  -d\right)^{-1}.
\end{equation*}
At the same time, in order to cover the remaining range of $\sigma \ge 1+ \left(d\beta (1+a_{\epsilon_1}(t_0)^{-1}) +d\right)/\log T_0$, we again use \eqref{tenen trivial}, and take the maximum of the two results.

Finally, putting everything together into Lemma \ref{lem_3.2_HB} and making the substitution $T_0\mapsto t$, we prove our result.
\end{proof}

\subsection{Combining both cases $t\le H$ and $t\ge H$}

\begin{proof}[Proof of Corollary \ref{cor main zeta'/zeta}]
For the case of $t\ge H$, with a fixed choice of $W$ as defined in \eqref{W define}, it remains to optimise $Q_H$ in Theorem \ref{main zeta'/zeta for H} over $d$, $\beta$, $\epsilon_1$, $\sigma_1$, and $\eta$, subject to the conditions mentioned in said theorem, while also checking that \eqref{tcond not RH1} holds, which is done in a similar way as in the proof of Corollary \ref{cor zeta'/zeta RH}. We found that the optimal value for $d$ in all cases is $d=1/W_0$, and \eqref{tcond not RH1} holds for all $t\ge t_0\ge H$. All computed parameters are presented in Table \ref{table:Q}, rounded to six decimal places.

On the other hand, a simple calculation shows that for $t\ge t_0 = 13$,
\begin{equation*}
1-\frac{1}{W_0 \log t} > 0.929.
\end{equation*}
Now, for $\sigma \ge 0.9$ and $13\le t\le H$, Corollary \ref{cor zeta'/zeta RH} and Table \ref{table:QRH} asserts that
\begin{equation*}
\left| \frac{\zeta'(\sigma+it)}{\zeta(\sigma+it)}\right| \le Q_{RH(H_0)}\log t = 13.123 \log t,
\end{equation*}
which is smaller than any value of optimised $Q_H$ found. Thus this same value of $Q_H$ holds for all $t\ge 13$.
\end{proof}

We also prove Corollary \ref{main zeta'/zeta for 1}, especially relevant for $\sigma= 1$.

\begin{proof}[Proof of Corollary \ref{main zeta'/zeta for 1}]
First note that Corollary \ref{cor zeta'/zeta RH} and Table \ref{table:QRH} has for $\sigma\ge 1$ and $13\le t\le H$,
\begin{equation}\label{compare Q1*}
\left| \frac{\zeta'(\sigma+it)}{\zeta(\sigma+it)}\right| \le Q_{RH(H_0)}\log t = 8.101 \log t.
\end{equation}

Then for $t\ge H$, the proof is identical to that of Theorem \ref{main zeta'/zeta for H} and Corollary \ref{cor main zeta'/zeta}, except we can restrict the size of $\alpha\beta r$ even further, since we can now set $\sigma_0 -\alpha \beta r = 1$, which is equivalent to 
\begin{equation}\label{new beta}
\beta = (1+a_{\epsilon_1}(t_0)^{-1})^{-1}.
\end{equation}
So in this case, for $\sigma \ge 1$ and $t\ge t_0 \ge H$, we have that
\begin{equation}\label{Q1*}
    \left| \frac{\zeta'(\sigma+it)}{\zeta(\sigma+it)}\right| \le \max\left\lbrace  \lambda_1 \left( \frac{1}{6}+\frac{2\log\log t_0}{\log t_0}+\frac{\log A_3}{\log t_0}\right)+\lambda_2,\, \frac{1}{2d}\right\rbrace \, \log t,
\end{equation}
where $\lambda_1$ and $\lambda_2$ are defined as in \eqref{lambda}, but with the added restriction of \eqref{new beta}.

Similar to the proof of Corollary \ref{cor main zeta'/zeta}, we then optimise \eqref{Q1*} over $d$, $\epsilon_1$, $\sigma_1$, and $\eta$, while ensuring that we fulfil the appropriate conditions in Theorem \ref{main zeta'/zeta for H} and \eqref{tcond not RH1}. The optimal value of $d$ found was again $d=1/W_0$. We then compare the result with \eqref{compare Q1*} and take the maximum to obtain a bound that holds for all $t\ge 13$.
\end{proof}

\section{Estimates for $|1 / \zeta(s)|$}\label{section 1/zeta}

Moving from a bound on the logarithmic derivative of $\zeta$ to one on the reciprocal of $\zeta$ is done in the usual way (see \cite[Theorem~3.11]{titchmarsh1986theory}). It revolves around estimating the right-hand side of 
\begin{align}\label{1/zeta_int0}
\log\left| \frac{1}{\zeta(\sigma+it)}\right| = - \textup{Re} \log \zeta\left( 1+ \delta_1 +it\right) + \int_\sigma^{1+\delta_1} \textup{Re}\left( \frac{\zeta'}{\zeta}(x+it)\right) \text{d}x
\end{align}
for some $\delta_1 >0$. 

The estimation of the first term on the right-hand side can be improved using the trigonometric polynomial (see \cite[Proposition~A.2]{carneiro2022optimality}, \cite{leongmossinghoffArxiv}), while the above integral is estimated by bounding the real part of the logarithmic derivative by a uniform bound on its absolute value $|\zeta'(s)/\zeta(s)|$. Here some savings can be made by using the method of \cite[Lemma~10]{cully2024explicit}), which splits the interval $[\sigma,1+\delta_1]$ and uses the sharpest bounds for $|\zeta'(s)/\zeta(s)|$ in the appropriate range, rather than one uniform bound for the entire interval.

Furthermore, through an argument utilising the trigonometric polynomial, we are able to get a slight asymptotic improvement on the bounds, at the cost of a worse constant.

\begin{theorem}\label{thm 1/zeta}
Let $t_0\ge 13$ and $d_1>0$ be a constant. Let $\{ W_{j} : 1\leq j\leq J\}$ be a sequence of increasing real numbers where $(W_j, Q_{1,j})$ is a pair for which $t\ge t_0$ and
\begin{equation*}
\sigma \ge 1- \frac{1}{W_j\log t} \quad\text{implies}\quad 
\left| \frac{\zeta'(\sigma +it)}{\zeta(\sigma+it)}\right| \le Q_{1,j}\log t.
\end{equation*}
Then for 
\begin{equation*}
    \sigma \ge 1- \frac{1}{W_1\log t},
\end{equation*}
we have
\begin{equation*}
\left| \frac{1}{\zeta(\sigma+it)}\right| \le Y_0(d_1,\sigma_1,t_0,\eta)\log t,
\end{equation*}
where 
\begin{align}\label{Y0}
Y_0&(d_1,\sigma_1,t_0,\eta) =   \max\Biggl\{ \frac{1}{d_1}+\frac{1}{\log t_0}, \\
& V(d_1/\log t_0,\sigma_1,t_0,\eta)\cdot \exp\Bigg( \sum_{j=1}^{J-1} Q_{1,j} \left( \frac{1}{W_j}-\frac{1}{W_{j+1}}\right) + \frac{Q_{1,J}}{W_J} +24.303d_1 \Bigg)\Biggr\rbrace , \nonumber
\end{align}
with any $1+d_1/\log t_0 \le \sigma_1 \le 2$, $\gamma$ denoting Euler's constant, and $V$ defined as in Lemma \ref{inverse zeta trig poly}.
\end{theorem}
\begin{proof}
Let $\delta_1 = \delta_1(t) := d_1/ \log t$. We will split the proof into two cases: $\sigma \le 1+\delta_1$ and $\sigma> 1+\delta_1$. First see that in the range $1- 1/W_1\log t\le\sigma \le 1+ \delta_1$ we have 
\begin{align}\label{1/zeta_int}
\log\left| \frac{1}{\zeta(\sigma+it)}\right| &= - \textup{Re} \log \zeta\left( 1+ \delta_1 +it\right) + \int_\sigma^{1+\delta_1} \textup{Re}\left( \frac{\zeta'}{\zeta}(x+it)\right) \text{d}x. 
\end{align}
Writing $\Delta = 1/\log t$, we can split the integral and rewrite it as
\begin{align}\label{smooth}
\left( \int_{1-\tfrac{\Delta}{W_1}}^{1-\tfrac{\Delta}{W_{2}}} + \cdots + \int_{1-\tfrac{\Delta}{W_j}}^{1-\tfrac{\Delta}{W_{j+1}}} +\cdots + \int_{1-\tfrac{\Delta}{W_J}}^{1} + \int_1^{1+\delta_1} \right) \textup{Re}\left( \frac{\zeta'}{\zeta}(x+it)\right) \text{d}x.
\end{align}
By Corollary \ref{main zeta'/zeta for 1} and our assumptions we have
\begin{align*}
\int_\sigma^{1+\delta_1} \textup{Re}\left( \frac{\zeta'}{\zeta}(x+it)\right) \text{d}x &\leq \sum_{j=1}^{J-1} Q_{1,j} \left( \frac{1}{W_j}-\frac{1}{W_{j+1}}\right) + \frac{Q_{1,J}}{W_J} +24.303d_1.
\end{align*}
We therefore have, from \eqref{1/zeta_int},
\begin{align}\label{1/zeta_int1}
\log\left| \frac{1}{\zeta(\sigma+it)}\right| <& -\log\left| \zeta\left( 1+\frac{d_1}{\log t}+ it \right)\right| \\
&+ \sum_{j=1}^{J-1} Q_{1,j} \left( \frac{1}{W_j}-\frac{1}{W_{j+1}}\right) + \frac{Q_{1,J}}{W_J} +24.303d_1 . \nonumber
\end{align}
To estimate the first term on the right-hand side of \eqref{1/zeta_int1} we apply Lemma \ref{inverse zeta trig poly}, keeping in mind our restrictions on $\sigma_1$. This gives us
\begin{equation}\label{V term}
\left| \frac{1}{\zeta(1 + \delta_1 +it)}\right| \le V(d_1/\log t_0,\sigma_1,t_0,\eta) \log t.
\end{equation}
We combine this with \eqref{1/zeta_int1} and exponentiate to obtain $Y_0$ for $\sigma \le 1+\delta_1$.

For the case $\sigma > 1+\delta_1$ we have from the trivial bound $|1/\zeta(s)|\le \zeta(\sigma)$ with \eqref{zeta trivial bd} that
    \begin{equation}\label{replace 2}
    \left| \frac{1}{\zeta(\sigma+it)}\right| \le \left(\frac{1}{d_1}+\frac{1}{\log t_0}\right)\log t,
    \end{equation}
since $\sigma/(\sigma-1)$ is decreasing in $\sigma$. 

Taking the maximum from both cases completes the proof.
\end{proof}

Corollary \ref{cor main 1/zeta} follows from Theorem \ref{thm 1/zeta} by computing $Y_0$ for specific choices of $t_0$ and $W_j$, and optimizing over $d_1$, $\sigma_1$, and $\eta$. These computed values of $Y_0$ are labelled $Y$, and are presented in Table \ref{table:Y}, with its corresponding parameters (rounded to six decimal places). The process is the same as that of \cite[Corollary 3]{cully2024explicit}, but for completeness sake, we repeat it here.

\begin{proof}[Proof of Corollary \ref{cor main 1/zeta}]
    After choosing $t_0 =13$, we aim to minimise $Y_0$ by optimising over $d_1>0$, $\eta$, and $\sigma_1$, while making sure that the conditions on these parameters are fulfilled. 
    
    This optimisation process first requires fixing the desired $W_1$, which then tells us the number of $W_j$ we are summing over, since from Table \ref{table:Q} we always have $(W_J,Q_{1,J})=(13,52.306)$. Computing more values of $(W_J,Q_{1,J})$ will give a better result, but the improvements eventually become negligible, as $Y_0$ is largely determined by the initial few $W_j$. For our computations, we only used values for $W_j =W$ from Table \ref{table:Q} (Corollary \ref{cor main zeta'/zeta}), thus we require $t_0 =13$.  

    Note that one can adjust Theorem \ref{thm 1/zeta} to compute $Y$ for the case $\sigma \ge 1$. In that case, we will modify \eqref{Y0} by setting
    \begin{equation*}
        \sum_{j=1}^{J-1} Q_{1,j} \left( \frac{1}{W_j}-\frac{1}{W_{j+1}}\right) + \frac{Q_{1,J}}{W_J} =0
    \end{equation*}
    before optimising. We obtain $Y = 30.812$ for $\sigma \ge 1$ and $t\ge 13$, with $(d_1,\sigma_1,\eta)=(0.032871,1.662479,3.216997)$. This proves the second assertion of the corollary.
    
    To prove the third assertion, we first compute $Y = 29.388$ for $\sigma \ge 1$ and $t\ge 2\exp(e^2)$, with $(d_1,\sigma_1,\eta)=(0.034172, 1.286253, 3.718165)$. Then see that from Lemma \ref{aleks lemma},
    \begin{equation*}
        \left| \frac{1}{\zeta(1+it)}\right| \le \frac{8}{\log(500)}\log t \le 1.288\log t \quad \text{for}\quad 500\le t\le 2\exp(e^2),
    \end{equation*}
    while a computation from the proof of \cite[Proposition~A.2]{carneiro2022optimality}, implemented in interval arithmetic, states that
    \begin{equation}\label{chirre calc}
    \left| \frac{1}{\zeta(1+it)}\right| \le 2.079\log t \quad \text{for}\quad 2\le t\le 500,
    \end{equation}
    so we are done.
\end{proof}

\begin{table}[ht!]
\centering
\setlength\tabcolsep{12pt}
\begin{tabular}{|c|c||c|c|c|} 
 \hline
 $W$ & $Y$ & $d_1$ & $\sigma_1$ & $\eta$ \\ \hline
 $5.56$ & $7.59\cdot 10^{31}$ & $ 0.030648 $ & $ 1.151072 $ & $ 3.150608 $ \\
 $6$ & $2.46\cdot 10^{9}$ & $ 0.030648 $ & $ 1.150464 $ & $ 3.150009 $ \\
 $7$ & $1.83\cdot 10^{6}$ & $ 0.030647 $ & $ 1.151694 $ & $ 3.149136 $ \\
 $8$ & $89357$ & $ 0.030647 $ & $ 1.148849 $ & $ 3.149156 $\\
 $9$ & $19482$ & $ 0.030647 $ & $ 1.150203 $ & $ 3.150832 $ \\
 $10$ & $7589$ & $ 0.030648 $ & $ 1.149868 $ & $ 3.150556 $ \\
 $11$ & $3972$ & $ 0.030650 $ & $ 1.150918 $ & $ 3.151137 $ \\
 $12$ & $2472$ & $ 0.030646 $ & $ 1.148656 $ & $ 3.150041 $ \\
 $13$ & $1719$ & $ 0.030647 $ & $ 1.149567 $ & $ 3.150198 $ \\
 \hline
\end{tabular}
\caption{Values for $Y$ in Corollary \ref{cor main 1/zeta}, with corresponding $W$, and parameters (rounded to six decimal places): $d_1$, $\sigma_1$, and $\eta$. Each entry is valid for $t\ge 13$.}
\label{table:Y}
\end{table}

\subsection{An asymptotic improvement}\label{asymp}

In this section we prove an asymptotic improvement over the bounds given earlier in \S\ref{section 1/zeta}, at the expense of a worse constant. For the sake of clarity, we first outline the general idea here, before proving the next theorem. 

Suppose we have the bound
\begin{equation*}
    \max_{1-\delta\le \sigma\le 1+\delta_1}\left| \Re \left( \frac{\zeta'}{\zeta}(\sigma+it)\right)\right| \le C\log t,
\end{equation*}
for $t\ge t_0$ and $1-\delta \le \sigma \le 1+\delta_1$, where $\delta = d/\log t$ and $\delta_1 = d_1/\log t$, as before. Then from \eqref{1/zeta_int0}, we have
\begin{align}\label{1/zeta_int2}
    \left| \frac{1}{\zeta(\sigma +it)}\right| &\le \left|\frac{1}{\zeta(1+\delta_1 +it)} \right|\exp\left((\delta_1+\delta)\max_{1-\delta\le \sigma\le 1+\delta_1}\left| \Re \frac{\zeta'}{\zeta}(\sigma+it)\right|\right)\nonumber \\
    &\le \left|\frac{1}{\zeta(1+\delta_1 +it)} \right| \exp\left((d_1+d)C \right).
\end{align}
Notice that since the exponential factor is $O(1)$, the overall order is determined by the first factor. 

With the same argument involving the trigonometric inequality \eqref{trig ineq},
\begin{equation}\label{trig ineq 1}
    \left| \frac{1}{\zeta(1 + \delta_1 +it)}\right| \le |\zeta(1 + \delta_1)|^{3/4}|\zeta(1 + \delta_1 + 2it)|^{1/4},
\end{equation}
and at the expense of a worse constant, we can estimate the second factor in \eqref{trig ineq 1} differently from Theorem \ref{thm 1/zeta}. We do this by using Lemma \ref{zeta_log_bound 2/3}, which states that $|\zeta(s)| \le C_3 (\log t)^{2/3}$ for $\sigma\ge 1$. Along with the second inequality in Lemma \ref{bastien bound} and our choice of $\delta_1$, this leads to
\begin{equation*}
    \left| \frac{1}{\zeta(1 + \delta_1 +it)}\right| \le (C_4\log t)^{3/4}(C_3 (\log t)^{2/3})^{1/4},
\end{equation*}
which applied to \eqref{1/zeta_int2} ultimately furnishes the bound
\begin{equation*}
     \left| \frac{1}{\zeta(\sigma +it)}\right| \le C_5(\log t)^{11/12}.
\end{equation*}

\begin{theorem}\label{1/zeta asymp}
Let the hypotheses of Theorem \ref{thm 1/zeta} be satisfied. Then for $t\ge t_0 \ge 13$ and
\begin{equation*}
    \sigma \ge 1- \frac{1}{W_1\log t},
\end{equation*}
we have
\begin{equation*}
\left| \frac{1}{\zeta(\sigma+it)}\right| \le Y'_0(d_1,t_0)(\log t)^{11/12},
\end{equation*}
where 
\begin{align*}
Y'_0 &(d_1,t_0) = C_3^{1/4}\cdot\max \Bigg\lbrace \left(\frac{1}{d_1} +\frac{1}{\log t_0}\right)^{3/4} , \\
&\exp\Bigg( \sum_{j=1}^{J-1} Q_{1,j} \left( \frac{1}{W_j}-\frac{1}{W_{j+1}}\right) + \frac{Q_{1,J}}{W_J} +24.303d_1 \Bigg) \left(\frac{e^{\gamma d_1/\log t_0}}{d_1}\right)^{3/4} \Bigg\rbrace, \nonumber
\end{align*}
with $d_1$, $J$, $Q_{1,j}$, $W_j$ defined as in Theorem \ref{thm 1/zeta}, and $C_3$ defined as in Lemma \ref{zeta_log_bound 2/3}.
\end{theorem}
\begin{proof}
The proof is identical to that of Theorem \ref{thm 1/zeta}, except that in the case $\sigma \le 1+\delta_1$, we replace \eqref{V term} with
\begin{align*}
    \left| \frac{1}{\zeta(1+\delta_{1}+it)}\right| &\le |\zeta(1+\delta_{1})|^{3/4} |\zeta(1+\delta_{1}+2it)|^{1/4}\\
    &\le \left(\frac{e^{\gamma(d_1/\log t_0)}}{d_1} \log t\right)^{3/4}\left( C_3(\log t)^{2/3}\right)^{1/4}
\end{align*}
by using \eqref{trig ineq 1} with Lemmas \ref{zeta_log_bound 2/3} and \ref{bastien bound}.

Meanwhile, for $\sigma > 1+\delta_{1}$, the same argument with \eqref{trig ineq} and \eqref{zeta trivial bd} leads to
\begin{equation*}
  \left| \frac{1}{\zeta(\sigma+it)}\right| \le \left(\left( \frac{1}{d_1}+\frac{1}{\log t_0}\right) \log t\right)^{3/4}\left( C_3(\log t)^{2/3}\right)^{1/4}
\end{equation*}
in place of \eqref{replace 2}.
\end{proof}

\begin{proof}[Proof of Corollary \ref{cor main 1/zeta asymp}]
    The proof, and process of optimisation here, even for the case $\sigma \ge 1$, are exactly the same as that of Corollary \ref{cor main 1/zeta}. These computed values of $Y'_0$ for $t\ge 13$, are labelled $Y'$, and are presented in Table \ref{table:Y'}. For all these cases, we found an optimal parameter of $d_1 = 0.030648$ (rounded to six decimal places).

    For the case $\sigma \ge 1$, we computed $Y' = 87.725$ for $t\ge 13$, also with optimal $d_1=0.030648$ (rounded).

    Finally, for the case $\sigma =1$, we computed $Y' = 81.900$ for $\sigma \ge 1$ and $t\ge 2\exp(e^2)$, with $d_1=0.030793$ (rounded). Then by a similar calculation as in the proof of Corollary \ref{cor main 1/zeta}, we have by Lemma \ref{aleks lemma},
    \begin{equation*}
        \left| \frac{1}{\zeta(1+it)}\right| \le 1.5(\log t)^{11/12} \quad \text{for}\quad 500\le t\le 2\exp(e^2).
    \end{equation*}
    On the other hand, \eqref{chirre calc} implies
    \begin{equation*}
    \left| \frac{1}{\zeta(1+it)}\right| \le 2.421(\log t)^{11/12} \quad \text{for}\quad 2\le t\le 500,
    \end{equation*}
    where this estimate is obtained by noting that $t$ attains its maximum at $500$.
\end{proof}
\begin{table}[ht!]
\centering
\setlength\tabcolsep{12pt}
\begin{tabular}{|c|c|} 
 \hline
 $W$ & $Y'$  \\ \hline
 $5.56$ & $2.17\cdot 10^{32}$ \\ 
 $6$ & $6.99\cdot 10^{9}$  \\ 
 $7$ & $5.21\cdot 10^{6}$  \\ 
 $8$ & $2.55\cdot 10^{5}$  \\ 
 $9$ & $55579$  \\ 
 $10$ & $21649$ \\ 
 $11$ & $11331$ \\ 
 $12$ & $7051$ \\ 
 $13$ & $4904$ \\ 
 \hline
\end{tabular}
\caption{Values for $Y'$ in Corollary \ref{cor main 1/zeta asymp}, with corresponding $W$. Each entry is valid for $t\ge 13$.}
\label{table:Y'}
\end{table}

\section{Concluding Remarks}

Here we state some possible ways to improve our estimates. The most impactful would be having a larger zero-free region or a higher Riemann height. Using an asymptotically stronger zero-free region is also helpful. See for instance \cite{cully2024explicit}, where a Littlewood-style zero-free region is used to obtain bounds of order $\log t/\log\log t$. Minor improvements can be made by a more careful treatment of the sum over zeros and the integral arising from the application of Lemma \ref{lem_3.2_HB}. For Theorem \ref{thm 1/zeta}, a better smoothing argument could be used in place of \eqref{smooth}.

Other possibilities for improvement are: a refinement on the Borel--Carath\'eodory by specialising it more to this problem, using stronger estimates for $\zeta(s)$ under $RH$ or $RH(T)$, or obtaining a good bound on $\Re (\zeta'(s)/\zeta(s))$ for use in \eqref{1/zeta_int0}.

It is evident from the results that the closer $s$ is to the boundary of the zero-free region, and hence a potential zero, the larger the constants. 
However, the large constants can also be attributed in part to the factor $(R+r)/(R-r)$ in Theorem \ref{borelcaratheodory}, which increases rapidly as $r$ approaches $R$. This contributes significantly to the overall estimate, even if $|f(s_0)|$ is small, and seems to suggest that Theorem \ref{borelcaratheodory} in its current form might not be optimal for bounding $|\zeta'(s)/\zeta(s)|$.

An advantage of obtaining a bound on $\Re (\zeta'(s)/\zeta(s))$ is to reduce the usage of Theorem \ref{borelcaratheodory}, when bounding $|1/\zeta(s)|$. This is because we would be able to apply such a bound directly to \eqref{1/zeta_int0}, whereas we currently use $\Re (\zeta'(s)/\zeta(s)) \le |\zeta'(s)/\zeta(s)|$. This alternative could be approached via Lemma \ref{lem_3.2_HB}, provided one were able to deal with the resulting sum over zeros well.

To get unconditional results we have (due to knowledge of the Riemann height) assumed a partial Riemann hypothesis to prove some of our results. The downside to this is that we have not used better available bounds that require the full strength of $RH$. For example, under $RH$ we could have used that $\zeta(1/2 +it) = O(\exp(C\log t/\log\log t))$ for some $C>0$ to prove something stronger (see \cite[\S14]{titchmarsh1986theory}, \cite{simonivc2022explicit}, \cite{Simonic22}, and \cite{chirre2024conditional} for more details). 

Even in the unconditional results used, there is room for improvement. As an illustration, Corollary \ref{cor main 1/zeta asymp}, though asymptotically superior, beats Corollary \ref{cor main 1/zeta} only when $t$ is large. In the case of $\sigma =1$, this happens for $t\ge \exp(2.2\cdot 10^{5})$. This is in part due to the relatively large constants in Corollary \ref{cor main 1/zeta asymp}, and a direct way of lowering these constants is an improvement in the bound $|\zeta(1+it)|$, which correspondingly lowers $C_3$ in Lemma \ref{zeta_log_bound 2/3}.

Finally, we conclude with a discussion on the trigonometric inequality \eqref{trig ineq} used in many of our above proofs, in particular when bounding $|1/\zeta(s)|$ for $\sigma >1$. Under certain conditions, this method is superior to using the trivial bound. In general, when studying zero-free regions for $\zeta(s)$, increasing the degree of the non-negative trigonometric polynomial improves the results one gets. Surprisingly, the same cannot be said in our setting, and the classical polynomial we used is actually optimal. Take for instance Lemma \ref{inverse zeta trig poly}, which is used in Theorem \ref{thm 1/zeta}. Here the overall order of the bound is $O(\log t)$, and in this scenario the article \cite{leongmossinghoffArxiv} has an in-depth discussion on why the classical polynomial is best.

On the other hand, in Theorem \ref{1/zeta asymp}, the question is whether higher degree polynomials lead to a reduction in the exponent $11/12$, even if overall constants get worse. This is an interesting question to consider, since in this case the answer is less obvious, although preliminary calculations suggest that the classical polynomial is still the optimal one.

\section*{Acknowledgements}

We thank Tim Trudgian for his guidance throughout the writing of this paper. We also thank the referee for the care taken in going through our article. Finally, we are very grateful to Andrew Fiori, Bryce Kerr, Olivier Ramar\'e, Aleksander Simoni\v{c}, for the helpful discussions, and especially to Harald Helfgott for pointing out an error in an earlier version of this paper.

\bibliographystyle{amsplain} 
\bibliography{references}

\end{document}